\documentclass[letterpaper, 10 pt, conference]{ieeeconf}

\IEEEoverridecommandlockouts
\usepackage{cite}
\usepackage{amsmath,amssymb,amsfonts}
\usepackage{algorithmic}
\usepackage{graphicx}
\usepackage{textcomp}
\usepackage{xcolor}
\usepackage[makeroom]{cancel}
\usepackage[T1]{fontenc}
\usepackage{mathtools}

\let\labelindent\relax

\usepackage{enumerate}
\usepackage{enumitem}

\usepackage{bbm}
\usepackage{amsthm}
\usepackage{amsmath}
\usepackage{thmtools}
\usepackage{tikz}
\usetikzlibrary{shapes.geometric, arrows.meta, positioning}
\usepackage{amsfonts}
\usepackage[thinc]{esdiff}
\usepackage{url}
\usepackage{hyperref}
\usepackage{cleveref}

\declaretheorem[]{lemma}
\declaretheorem[name = Theorem]{thm}

\declaretheorem[name = Assumption]{assump}
\declaretheorem[name = Remark, style = remark]{rmk}

\tikzstyle{block} = [rectangle, draw, fill=blue!20, text width=2.5em, text centered, rounded corners, minimum height=2em]
\tikzstyle{line} = [draw]
\tikzstyle{input} = [coordinate]
\tikzstyle{output} = [coordinate]

\newcommand{\R}{\mathbb{R}}

\newcommand{\Vv}{\mathcal{V}}

\newcommand{\gzero}{\gamma_0}

\newcommand{\gone}{\gamma_1}
\newcommand{\snext}{s^{+}}
\newcommand{\xnext}{x^{+}}
\newcommand{\dt}{u}
\newcommand{\ett}{\eta(u)}
\newcommand{\etzero}{\eta_0 }
\newcommand{\etone}{\eta_1 }
\newcommand{\ettwo}{\eta_2 }

\newcommand{\gmax}{\|g\|_{\infty{}}}
\newcommand{\smax}{\| s\|_{\infty}}
\newcommand{\dmax}{\| y \|}
\newcommand{\G}{\| g\|_{\infty}}

\newcommand\xmax{\| y \|_{\infty}}
\newcommand{\lin}{l}
\newcommand{\nlin}{p}
\newcommand{\dst}{u}
\newcommand{\Uu}{\mathcal{U}_{\geq0}}

\newcommand{\rrd}{\R^d \times \R_{\geq 0}^d}

\makeatletter
\newcommand{\distas}[1]{\mathbin{\overset{#1}{\kern\z@\sim}}}%
\newsavebox{\mybox}\newsavebox{\mysim}

\def\BibTeX{{\rm B\kern-.05em{\sc i\kern-.025em b}\kern-.08em
    T\kern-.1667em\lower.7ex\hbox{E}\kern-.125emX}}
\begin{document}

\title{Global Stability and Step Size Robustness of RMSProp
}

\author{Naum Dimitrieski, Maria Christine Honecker, Carsten Scherer, and Christian Ebenbauer
\thanks{Naum Dimitrieski, Maria Christine Honecker, and Christian Ebenbauer are with the Chair of Intelligent Control Systems, RWTH Aachen University, 52062 Aachen, Germany
        {\tt\small (e-mail: \{naum.dimitrieski, maria.honecker, christian.ebenbauer\}@ic.rwth-aachen.de)}, and  
        Carsten Scherer is with the Chair of Mathematical Systems Theory, University of Stuttgart, 70569 Stuttgart, Germany {\tt\small (e-mail: \{carsten.scherer\}@mathematik.uni-stuttgart.de)}
        }%
}

\maketitle

\begin{abstract}
In this paper, an input-to-state Lyapunov function for the RMSProp optimization algorithm is introduced. Global asymptotic stability of the RMSProp algorithm for constant step sizes and robustness properties with respect to arbitrary bounded time-varying step size rules are established. 
\end{abstract}
%

\section{Introduction}

Adaptive gradient methods are gradient-based optimization methods where the step size is adapted in a feedback fashion
based on current and past iterates and gradients. They have proven to be effective in many fields of optimization and machine learning \cite{defossez2020simple}. Among the most well-known adaptive gradient methods are AdaGrad \cite{duchi2011adaptive}, RMSProp \cite{tieleman_hinton2012lecture} and Adam \cite{kingma2014adam}. 
As surveyed in \cite{abdulkadirov2023survey, he2021survey_Adam}, Adam and Adam-like methods are central to training deep neural networks. Notably, Adam integrates adaptive step sizes from RMSProp with momentum in the Polyak sense \cite{polyak1964some}. Thus, RMSProp represents a core component of Adam and its variants.

While momentum methods are quite well understood and in recent years rather extensively analyzed from a systems theoretic perspective (see, e.g.,\cite{schererebenbauer2025_tutorialIQC}), adaptive step size rules remain less well understood.
Hence, the purpose of this work is to establish basic systems-theoretic properties of the deterministic RMSProp algorithm
in the sense of global stability and robustness.

In recent years, considerable research effort has been devoted to establishing convergence and stability properties of RMSProp. Due to the close relation between Adam and RMSProp (see, e.g., \cite{chen2022towards, zou2019sufficient}) we also review the literature related to Adam, in particular, works whose results apply to RMSProp as well.
We start with local stability properties of RMSProp, which have been analyzed in recent works. For instance, in \cite{dereich2025sharp_conv_rate_RMSProp} local exponential stability is established, while in \cite{bensaid2024abstract} local asymptotic stability under a backtracking line-search step size rule is shown.

The literature on global convergence properties of RMS\-Prop can be broadly divided into deterministic results (deterministic gradients), and  stochastic results (stochastic gradients). For instance, in the deterministic setting, \cite{barakatRMSProp2020convergence} enforces a modified step size rule guaranteeing  sufficiently small adaptive step sizes, while \cite{bensaid_Lyapunov2024convergence} achieves a similar result under a backtracking line-search step size rule. In both works, convergence to a stationary point is ensured by enforcing a sufficiently small step size. 

Global stability results are established, for instance, in \cite{ bensaid2023deterministic,heredia2024modeling , dereich2025asymptotic_stability}. In particular, \cite{bensaid2023deterministic, heredia2024modeling} provide stability results on the continuous-time algorithm.
Further, \cite{dereich2025asymptotic_stability} shows boundedness of the iterates for a special class of objective functions which  does not encompass all strongly convex and L-smooth objective functions. Moreover, no result on the convergence of the iterates to any stationary point of the algorithm is provided.
Hence, global asymptotic stability of RMSProp and boundedness of iterates for arbitrary step size rules is not well explored in the literature.

In this paper, we establish novel systems-theoretic properties of RMSProp.
In more detail, the contributions are as follows. First, we propose a novel Lyapunov function and show global asymptotic 
stability for strongly convex and $L$-smooth objective functions.
Second, we show that this function is also an input-to-state Lyapunov function with respect to bounded time-varying step sizes. 
This property implies boundedness of the iterates for an arbitrary bounded step size rule. To the best of our knowledge,
the proposed Lyapunov function is the first global and input-to-state Lyapunov function for RMSProp. 
Our results allow, for example, 
to design bounded time-varying step size rules with guaranteed bounded iterates and allow to
control the trade-off between convergence rate and convergence accuracy for RMSProp.

\textit{Notation:} We denote the sets of real and non-negative numbers as $\R$ and $\R_{\geq 0}$, respectively, and the set of non-negative integers as~${\mathbb{N}_0}$. We denote by $\mathcal{C}^n$ the set of $n$ times continuously differentiable functions, and for any univariate function $h \in \mathcal{C}^{1}$, we denote the first derivative as~$h'$. Further, we say a function~$h$ belongs to class~${\mathcal{K}_{\infty}}$ if it satisfies \cite[Def. 2]{kellett2014compendiumISS}. For~$x \in \R^d$, we denote  the Euclidean norm by~${\| x \|}$.

\section{Problem Statement}
We consider an  optimization problem of the form
\begin{align}
    \label{problem:1}
    \underset{x \in \R^d}{\mathrm{min}}~ f(x),
\end{align}
where $f\colon\R^d \to \R$ is the objective function. The main goal of this paper is to analyze the stability and robustness of the RMSProp algorithm, as introduced in \cite{tieleman_hinton2012lecture}, where for each~${i \in \overline{1:d}}$ one defines  
\begin{subequations}
    \label{sys:rmsprop_original}
    \begin{align}
        \label{sys:s_original}
            s_{i}(t+1) &= (1-\beta)s_{i}(t) + \beta (\nabla_i f(x(t)))^2, \\
        \label{sys:x_original}
            x_{i}(t+1) &= x_{i}(t) - \tfrac{\eta(\dst(t))}{\varepsilon + \sqrt{s_{i}(t+1)}}\nabla_i f(x(t)),
    \end{align}
\end{subequations}
with the initial conditions $x(0) \in \R^d$, $s(0) \in \R^d_{\geq 0}$ as well as the parameters~${\varepsilon > 0}$,~${\beta \in (0,1)}$, and the step sizes~${\eta(\dst(t))\coloneqq \etzero  + \dst(t)}$, for ${\etzero  > 0}$ and
\begin{align*}
 \dst(\cdot) \in \Uu  \coloneqq  \big\{u\colon\mathbb{N}_0\to\mathbb{R}_{\geq0} \ \colon \ \text{$u(\cdot)$ is bounded}\big\}.
\end{align*}
In our analysis, we will choose $\etzero$ such that the algorithm is globally asymptotically stable for $u(\cdot)$ identically zero.
The term $u(\cdot)$ is introduced as an additional degree of freedom to tune the algorithm in terms of convergence rate and convergence accuracy (error floor).
To simplify the notation in \eqref{sys:rmsprop_original}, we omit the dependence on~$t$ and introduce the functions~${\snext\colon\rrd \to \R_{\geq 0}^d}$ and~${\xnext\colon\rrd\times \R_{\geq 0} \to \R^d}$ as
\begin{subequations}
    \label{sys:rmsprop}
    \begin{align}
        \label{sys:s}
            \snext_{i} &\coloneqq (1-\beta)s_{i} + \beta (\nabla_i f(x))^2, \\
        \label{sys:x}
            \xnext_{i} &\coloneqq x_{i} - \tfrac{{\ett}}{\varepsilon + \sqrt{\snext_{i}}}\nabla_i f(x).
    \end{align}
\end{subequations}
We make the following assumptions for our analysis.
\begin{assump}
    \label{assump}
    Consider problem  \eqref{problem:1}. Assume that $f$ is
    \begin{enumerate}[label= \roman*)]
        \item \label{assump:convex} $\mathcal{C}^1$, $\mu$-strongly convex for some $\mu > 0$, and has a unique global minimizer~${x^{*} \in \R^d}$, and
        \item \label{assump:L_smooth}
        globally $L$~smooth for some ${L > 0}$, i.e., for any~${x,y\in \R^d}$ we have~${\lVert \nabla f(y) - \nabla f(x) \rVert \!\leq\! L \lVert y - x \rVert}$.
    \end{enumerate}
\end{assump}
\Cref{assump} is restrictive but still standard in the optimization literature
for the purpose of algorithm analysis \cite{nesterov2013introductory}.

\section{Main Results}
In this section we present the main results of this paper, namely a Lyapunov function for algorithm \eqref{sys:rmsprop} and input-to-state stability (ISS) of this algorithm with respect to the step size rule ${\dt}(\cdot)$. 

We start with the Lyapunov function candidate for algorithm \eqref{sys:rmsprop}. We propose the Lyapunov function candidate~${\Vv\colon \R^d \times \R^d_{\geq 0}} \to \R_{\geq 0}$ 
\begin{align}
    \label{function:Lyapunov}
    \Vv(x,s) &\coloneqq \gamma(f(x) - f(x^{*}))  + 2\sum\limits_{i=1}^d h(s_i),
\end{align}
where for any $\omega \in \R_{\geq 0}$ the functions~$\gamma$ and~$h$ are defined as
\begin{subequations}
\label{lyapunov:subfunctions}
\begin{align}
    \label{lyapunov:gamma} \gamma(\omega)&\coloneqq \gzero\omega + \tfrac{2}{3}\gone\omega^{\frac{3}{2}}, \\
    \label{lyapunov:theta} h(\omega) &\coloneqq \sqrt{\omega} + \varepsilon\log(\varepsilon) - \varepsilon\log(\sqrt{\omega} + \varepsilon)
\end{align}
\end{subequations}
with the constants
\begin{align}
    \label{gzero_gbar}
    \gzero \! &\coloneqq \! \mathrm{max}\big\{3\beta\etzero^{-1} ,\gamma_{01}, \gamma_{02}\big\}, &   \gone \! &\coloneqq \!\tfrac{12 \beta^{\frac{3}{2}} \sqrt{L}}{\etzero  \varepsilon},
\end{align} with ${\gamma_{01}\coloneqq\tfrac{2\beta\varepsilon}{\etzero \left(2\varepsilon-(\etzero + \etone)  L\right)}}$,~${\gamma_{02} \coloneqq \tfrac{\gone \etzero  \sqrt{L d}}{\sqrt{\beta}} \! + \!\tfrac{12 \beta L\sqrt{d}}{\varepsilon}}$, and where~${\beta,\varepsilon}$ and~$\etzero$ are the parameters in algorithm~\eqref{sys:rmsprop},~$L$ is the smoothness constant of~$f$ in~\Cref{assump}~\ref{assump:L_smooth} and~${\etone > 0}$ is any positive constant.

Next, we provide the main theorem of the paper, formally characterizing the asymptotic stability and ISS for any input ${\dt(\cdot) \in \Uu}$ of algorithm \eqref{sys:rmsprop} with respect to the equilibrium~${(x^{*},0)}$. 

\begin{thm}
    \label{ISS-theorem}
    Consider algorithm \eqref{sys:rmsprop} and let Assumption \ref{assump} be satisfied with~${\mu > 0}$, ${x^{*}\in \R^d}$ and~${L > 0}$. Let ${\beta \in (0,1)}$, ${\varepsilon > 0}$, and~${\ett\coloneqq \etzero + \dt}$, where~${\etzero  \in \left(0, \tfrac{2\varepsilon}{L}\right)}$ and~${\dst(\cdot) \in \Uu}$. Moreover, let~${\etone > 0}$ satisfy~${\etzero + \etone < \frac{2\varepsilon}{L}}$. Then,
    algorithm \eqref{sys:rmsprop} is ISS for any step size ${\dst(\cdot) \in \Uu}$ with respect to the equilibrium~$(x^{*},0)$, and the function \eqref{function:Lyapunov} is an ISS-Lyapunov function for algorithm~\eqref{sys:rmsprop}.
\end{thm}

\Cref{ISS-theorem} implies that, given any choice of positive bounded step size rule, the iterates remain bounded. For a practitioner, this allows step size tuning without relying on fall-back techniques (e.g., projections onto a compact set) while balancing convergence rate and convergence accuracy (see \Cref{remark:2:trade_off}).

In the next lemma, we first show that the function \eqref{function:Lyapunov} satisfies the requirements for a Lyapunov function candidate, as stated in \cite[Definition 3.2., Condition 1]{jiang2001iss_discrete} and \cite[Theorem 3.4.6.]{michel2008stability}. 

\begin{lemma}
    \label{lemma:Lyapunov_fcn}
    Consider \eqref{function:Lyapunov}, $\gamma$ and~$h$ as defined in \eqref{lyapunov:subfunctions} and let
    Assumption \ref{assump} be satisfied with~${\mu > 0}$, ${x^{*}\in \R^d}$ and~${L > 0}$.  
    Moreover, let~${\varepsilon > 0}$,~${\beta \in (0,1)}$ and~${\etzero \in \left(0, \frac{2\varepsilon}{L}\right)}$, and let~${\etone > 0}$ satisfy~${\etzero + \etone < \frac{2\varepsilon}{L}}$. Then, there exist some~${\hat{\alpha}_1, \hat{\alpha}_2 \in \mathcal{K}_{\infty}}$ such that, for all~${(x,s) \in \rrd}$,
    \begin{align*}
        \hat{\alpha}_1(\| (x-x^{*},s)\|_{\infty}) \leq \Vv(x,s) \leq \hat{\alpha}_2(\| (x-x^{*},s)\|_{\infty}).
    \end{align*}
\end{lemma}
\begin{proof}
    We start with the properties of~${\gamma}$ in \eqref{lyapunov:subfunctions}. 
    By \eqref{gzero_gbar},~${\gone}$ is positive. Moreover, since by assumption we have~${\etzero  \in \left(0, \tfrac{2\varepsilon}{L}\right)}$ and~${\etzero + \etone < \frac{2\varepsilon}{L}}$, we infer~${\varepsilon - \tfrac{(\etzero + \etone) L}{2} > 0}$ and, thereby, the second term in the~${\mathrm{max}}$ operator in~${\gzero}$ in \eqref{gzero_gbar} is well defined. Consequently, we also have~${\gzero > 0}$.
    Then,~${\gamma(0) = 0}$, and in addition,~$\gamma$ is continuous, positive definite and radially unbounded on~${[0,\infty)}$. 
    
    Under Assumption~\ref{assump}~\ref{assump:convex}, we observe that~${f(x) - f(x^{*})}$, and thus~${\gamma(f(x) - f(x^{*}))}$, are continuous, positive definite in~${x \in \R^d}$ and radially unbounded. 
    Then, \cite[Lemma 4.3]{khalil2002nonlinear} implies that there exist some~${\alpha_{\gamma,m}, \alpha_{\gamma,M} \in \mathcal{K}_{\infty}}$ such that for all~${x \in \R^d\colon}$
    \begin{align}
        \label{gamma:upper_lower_comparison}
        \alpha_{\gamma,m}(\| x-x^{*}\|_{\infty}) &\leq \gamma(f(x) - f(x^{*})) \notag \\ &\leq \alpha_{\gamma,M}(\| x-x^{*}\|_{\infty}).
    \end{align}

    Next, we proceed with the properties of~$h$, defined in \eqref{lyapunov:subfunctions}. We observe that~${h(0) = 0}$ and that~${h}$ is continuous on~${[0,\infty)}$. For all~${\omega \in \R_{\geq 0}}$ we have~${h'(\omega) = \tfrac{1}{2(\omega  + \varepsilon)} > 0}$, i.e., $h$ is strictly monotonically increasing and hence,~$h$ is positive definite on~$[0,\infty)$. Finally, \eqref{lyapunov:theta} implies that the positive definite term~${\sqrt{\cdot}}$ grows at a faster rate than~${\varepsilon\log(\varepsilon) -\varepsilon\log(\sqrt{\cdot} + \varepsilon)}$ on~${[0,\infty)}$, which implies that~$h$ is radially unbounded. It therefore follows that~${h \in \mathcal{K}_{\infty}}$. 

    In the following we find suitable comparison functions for~${\Vv}$, defined in \eqref{function:Lyapunov}. We start by using the inequality
    \begin{align}
        \label{lower_bound_2}
        2\sum\limits_{i=1}^{d}h(s_i) \geq h(\|s\|_{\infty}),
    \end{align}
    which holds since~$h$ is  positive definite. By plugging the lower bounds from \eqref{gamma:upper_lower_comparison} and \eqref{lower_bound_2} into \eqref{function:Lyapunov}, one gets for all~${(x,s) \in \rrd}$ that
    \begin{align}
        \label{V:lower_bound}
        \Vv(x,s) &\geq \alpha_{\gamma,m}(\| x-x^{*}\|_{\infty}) + h(\smax)
        \notag \\ 
        &\geq \mathrm{min}\{\alpha_{\gamma,m}\left(\|(x-x^{*},s)\|_{\infty}\right), h(\|(x-x^{*},s)\|_{\infty}\},
    \end{align}
    where we used~${\mathrm{max}\{\|x-x^{*}\|_{\infty},\| s\|_{\infty}\} = \|(x-x^{*},s)\|_{\infty}}$.  The second inequality in \eqref{V:lower_bound} can be shown by a simple case analysis and is left for the reader to check.

    Next, we use the inequality
    \begin{align}
    \label{V:upper_bound_help_0}
        2\sum\limits_{i=1}^{d}h(s_i) &\leq 2dh(\|s\|_{\infty}),
    \end{align}
    which follows since~$h$ is positive definite and monotone on~${[0,\infty)}$.

    By plugging the upper bounds from \eqref{gamma:upper_lower_comparison} and \eqref{V:upper_bound_help_0} into~\eqref{function:Lyapunov}, we get for all~${(x,s) \in \rrd}$ that
    \begin{align}
        \label{V:upper_bound}
        \Vv(x,s) &\leq \alpha_{\gamma,M}(\|(x-x^{*},s)\|_{\infty})  + 2dh(\|(x-x^{*},s)\|_{\infty}).
    \end{align}
    Here, we used that the upper bounds \eqref{gamma:upper_lower_comparison} and \eqref{V:upper_bound_help_0} are strictly monotonically increasing and that
    \begin{align*}
        \|x-x^{*}\|_{\infty} &\leq \mathrm{max}\{\|x-x^{*}\|_{\infty},\| s\|_{\infty}\} = \|(x-x^{*},s)\|_{\infty}, \\
        \| s \|_{\infty} &\leq \mathrm{max}\{\|x-x^{*}\|_{\infty},\| s\|_{\infty}\} = \|(x-x^{*},s)\|_{\infty}.
    \end{align*}
    The lower and upper bounds in \eqref{V:lower_bound} and \eqref{V:upper_bound} belong to class~${\mathcal{K}_{\infty}}$. This is true, since both the minimum and the sum of~${\mathcal{K}_{\infty}}$-functions belong to class~${\mathcal{K}_{\infty}}$. This concludes the proof.
\end{proof}

In the following we provide a proof of \Cref{ISS-theorem}. 

\begin{proof}

    This proof is divided into three main parts. In the first part, we derive suitable preliminary upper bounds for the terms in the Lyapunov function difference. In the second part, we provide a refined form of the upper bound of the Lyapunov function difference which depends solely on~$\smax$,~${\| \nabla f(x) \|_{\infty}}$ and~${\dst \in \R_{\geq 0}}$. In the third part, starting from the upper bound from the second part, for algorithm \eqref{sys:rmsprop} we initially show that global asymptotic stability holds for~${\dst = 0}$ with respect to the equilibrium~${(x^{*},0)}$, by invoking \cite[Theorem 3.4.6.]{michel2008stability}.  
    Subsequently, we show that the Lyapunov function satisfies an input-to-state stability condition for any~${\dst \in \R_{\geq 0}}$, in the sense of an adaptation of \cite[Definition 3.2.]{jiang2001iss_discrete}, with respect to the equilibrium~${(x^{*},0)}$. Thereby, we get that~$\Vv$, defined in \eqref{function:Lyapunov}, is an ISS-Lyapunov function, and by invoking \cite[Lemma 3.5.]{jiang2001iss_discrete} we establish that \eqref{sys:rmsprop} is ISS for any~${\dst(\cdot) \in \Uu}$ with respect to the equilibrium~${(x^{*},0)}$.

    We start by defining the proof preliminaries. 
    Throughout the proof, for the sake of clarity, for any~${x \in \R^d}$ we use the notation
    \begin{align*}
        f_0(x)&\coloneqq f(x) - f(x^{*}), & g &\coloneqq \nabla f(x), &g_{i}\coloneqq\nabla_i f(x).
    \end{align*}
    Further, we denote the quadratic, arithmetic and geometric means as $\mathrm{QM}$, $\mathrm{AM}$ and $\mathrm{GM}$. 
    For any~${(x,s) \in \rrd}$, the Lyapunov function difference is given as
    \begin{align}
        \label{proof:theorem:difference_Lyapunov}
        & \Delta \Vv \coloneqq \Vv(\xnext, \snext) - \Vv(x,s) \notag \\
        =&\gamma(f_0(\xnext)) - \gamma(f_0(x)) + 2\sum\limits_{i=1}^{d}\big(h(\snext_i) - h(s_{i})\big),
    \end{align}
    with~${\snext}$ and~${\xnext}$ as defined in \eqref{sys:rmsprop}.
    
    Initially, we apply the mean value theorem on~${\gamma}$ as in~\eqref{lyapunov:gamma} for the case of~${f_0(\xnext) \neq f_0(x)}$. For all such~${(x,s) \in \rrd}$ and~${\dst \in \R_{\geq 0}}$, there exists some
    \begin{align*}
        \xi \in \mathcal{I}(x,s,u) \coloneqq (a_{\mathrm{min}}(x,s,u), a_{\mathrm{max}}(x,s,u)),
    \end{align*}
    where~${a_{\mathrm{min}}(x,s,u)\coloneqq\mathrm{min}\big\{f_0(x), f_0(\xnext)\big\} \geq 0}$, ${a_{\mathrm{max}}(x,s,u)\coloneqq\mathrm{max}\big\{f_0(x), f_0(\xnext)\big\} \geq 0}$, such that 
    \begin{align}
    \label{def:xi}
        \gamma'(\xi) = \tfrac{\gamma(f_0(\xnext)) - \gamma(f_0(x))}{f_0(\xnext) - f_0(x)}.
    \end{align}

    If~${f_0(\xnext) = f_0(x)}$, we set~${\xi = f_0(x) = f_0(\xnext) \geq 0}$. By~\eqref{lyapunov:gamma} and~\eqref{gzero_gbar}, we then have~${\gamma'(\xi) > 0}$, and the following trivially holds:
    \begin{align}
        \label{def:xi:2}
        \gamma(f_0(\xnext)) - \gamma(f_0(x)) = \gamma'(\xi)\left(f_0(\xnext) - f(x)\right) = 0.
    \end{align}
    Therefore, for any~${(x,s) \in \rrd}$ and any~${\dst \in \R_{\geq 0}}$ there exists some
    \begin{align}
        \label{def:xi:final}
        0 \leq \xi  \in \mathcal{I}(x,s,u) \cup\{f_0(x), f_0(\xnext)\}
    \end{align}
    such that, by using \eqref{def:xi} and~\eqref{def:xi:2}, we get from~\eqref{proof:theorem:difference_Lyapunov} that
    \begin{align}
        \label{proof:theorem:difference_Lyapunov:2}
        \Delta \Vv &=\gamma'(\xi )(f_0(\xnext) - f_0(x)) 
         + 2\sum\limits_{i=1}^{d}\big(h(\snext_{i}) - h(s_{i})\big).
    \end{align}

    \textbf{Part 1:} In this part we derive preliminary upper bounds for \eqref{proof:theorem:difference_Lyapunov:2}, which consist of negative definite and~${\dst}$-dependent components, where~${\dst \in \R_{\geq 0}}$.
    
    \textbf{Step 1.1:} In the following we derive upper bounds for the first right hand side term and each term of the right hand side sum of \eqref{proof:theorem:difference_Lyapunov:2}. We start by using the $L$-smoothness property of~$f$. Due to \Cref{assump}~\ref{assump:L_smooth} and that~${\gamma'(\xi) > 0}$ for all~${\xi  \geq 0}$ (see \eqref{lyapunov:gamma} and~\eqref{gzero_gbar}, and \eqref{def:xi:final}), it is possible to use 
    \cite[Theorem 2.1.5, (2.1.9)]{nesterov2013introductory} to get
    \begin{align*}
        \gamma'(\xi )(f_0(\xnext) - f_0(x)) &\leq \gamma'(\xi )\nabla f_0^{\top}(x)(\xnext - x) \\ &+ \gamma'(\xi )\tfrac{L}{2} \lVert \xnext - x \rVert_2^2.
    \end{align*}
    As~${\nabla f_0^{\top}(x)(\xnext - x) = \sum_{i=1}^{d}g_{i}(\xnext_{i} - x_{i})}$, we can subsequently use \eqref{sys:x} for~${\xnext_i - x_{i}}$ to get
    \begin{align}
        \label{proof:theorem:L_smoothness_ineq}
        \gamma'(\xi )(f_0(\xnext) - f_0(x)) &\leq \gamma'(\xi )\sum\limits_{i=1}^{d}g_{i}\left(- \tfrac{{\ett} g_{i}}{\varepsilon + \sqrt{\snext_{i}}}\right) \notag \\ &+  \gamma'(\xi )\tfrac{L}{2}\sum\limits_{i=1}^{d}\tfrac{{\ett}^2 g_{i}^2}{\left(\varepsilon + \sqrt{\snext_{i}}\right)^2}.
    \end{align}
   Next, we use the concavity of~$h$, defined in \eqref{lyapunov:theta}, for non-negative arguments to derive an upper bound for each summand in \eqref{proof:theorem:difference_Lyapunov:2}. By using this property, we get for every~${i \in \overline{1:d}}$ that
    \begin{align*}
        h(\snext_i) - h(s_{i}) \leq  h'(s_{i})(\snext_{i} - s_{i}),
    \end{align*}
    where $h'(s_{i}) =\tfrac{1}{2(\varepsilon + \sqrt{s_{i}})}$. By using \eqref{sys:s}, this is rewritten as
    \begin{align}
        \label{proof:theorem:concavity_ineq}
        h(\snext_{i}) - h(s_{i}) \leq \tfrac{-\beta s_{i} + \beta g_{i}^2}{2(\varepsilon + \sqrt{s_{i}})}.
    \end{align}
    
    \textbf{Step 1.2:} In the following, we plug the upper bounds from Step 1.1, i.e., \eqref{proof:theorem:L_smoothness_ineq} and \eqref{proof:theorem:concavity_ineq} into~\eqref{proof:theorem:difference_Lyapunov:2}, and subsequently reorganize the terms to get
    \begin{align*}
        \Delta \Vv &\leq -\sum\limits_{i=1}^{d}\tfrac{\beta s_{i}}{\varepsilon + \sqrt{s_{i}}} + \sum\limits_{i=1}^{d}\tfrac{g_{i}^2}{\left(\varepsilon + \sqrt{\snext_{i}}\right)^2(\varepsilon + \sqrt{s_{i}})} \notag \\ &\times \Bigg[-{\ett}\gamma'(\xi)(\varepsilon + \sqrt{s_{i}})\left(\varepsilon + \sqrt{\snext_{i}}\right) \notag \\ &+ \tfrac{{\ett}^2\gamma'(\xi)L}{2}\left(\varepsilon + \sqrt{s_{i}}\right) + \beta\left(\varepsilon + \sqrt{\snext_{i}}\right)^2\Bigg] .
    \end{align*}
   For each~${i \in \overline{1:d}}$, we now define 
    \begin{align*}
        \kappa_i &\coloneqq  -{\ett}\gamma'(\xi)(\varepsilon + \sqrt{s_{i}})\left(\varepsilon + \sqrt{\snext_{i}}\right) \\ &+ \tfrac{{\ett}^2\gamma'(\xi)L}{2}(\varepsilon + \sqrt{s_{i}}) + \beta\left(\varepsilon + \sqrt{\snext_{i}}\right)^2,
    \end{align*}
    and rearrange as
    \begin{align}
        \label{proof:theorem:K:0}
        \kappa_{i}&= -\sqrt{\snext_{i}}\left({\ett}\gamma'(\xi)\sqrt{s_{i}} - \beta\sqrt{\snext_{i}}\right) \notag \\ &- {\ett}\gamma'(\xi )\sqrt{s_{i}}\left(\varepsilon -\tfrac{\ett L }{2}\right) \notag \\& - 2\tfrac{\varepsilon}{2}\sqrt{\snext_{i}}{\ett}\gamma'(\xi )\left(1 - \tfrac{2\beta}{{\ett} \gamma'(\xi)}\right) \notag \\ &-\varepsilon\left({\ett}\gamma'(\xi )\left(\varepsilon - \tfrac{\ett L}{2}\right) - \beta\varepsilon\right).
    \end{align}
    Then, we get that 
    \begin{align}
        \label{proof:delta_V_ineq}
        \Delta \Vv \leq -\sum\limits_{i=1}^{d}\tfrac{\beta s_{i}}{\varepsilon + \sqrt{s_{i}}} + \sum\limits_{i=1}^{d}\tfrac{g_{i}^2\kappa_{i}}{\left(\varepsilon + \sqrt{\snext_{i}}\right)^2(\varepsilon + \sqrt{s_{i}})}.
    \end{align}
    The inequality \eqref{proof:delta_V_ineq} is an important intermediate result. 
    For the remainder of the proof, we derive suitable upper bounds for the terms on the right hand side of \eqref{proof:delta_V_ineq}.
    Also, observe that, in \eqref{proof:delta_V_ineq}, the first sum on the right hand side is negative definite in~${s \in \R_{\geq 0}^d}$.
    In particular, in the remainder of this part of the proof, we focus on the terms~${\kappa_i}$,~${i \in \overline{1:d}}$, in \eqref{proof:theorem:K:0}.
    For convenience, we define 
     $${\kappa_i = a^{i}_{11} + a^{i}_{21} + 2a^{i}_{31}\sqrt{\snext_i} +\varepsilon a^{i}_{41}}$$ for each~${i \in \overline{1:d}}$ with
    \begin{subequations}
    \label{terms_a_i}
    \begin{align}
        a^{i}_{11} & \coloneqq -\sqrt{\snext_i}\left({\ett}\gamma'(\xi)\sqrt{s_{i}} - \beta\sqrt{\snext_{i}}\right), \\
        a^{i}_{21} &\coloneqq - {\ett}\gamma'(\xi )\left(\varepsilon -\tfrac{\ett L }{2}\right)\sqrt{s_i}, \\
        a^{i}_{31} &\coloneqq -{\ett}\gamma'(\xi )\tfrac{\varepsilon}{2}\left(1 - \tfrac{2\beta}{{\ett} \gamma'(\xi)}\right), \\
        a^{i}_{41} &\coloneqq -{\ett}\gamma'(\xi )\left(\varepsilon - \tfrac{\ett L}{2}\right) + \beta\varepsilon.
    \end{align}
    \end{subequations}
      
    \textbf{Step 1.3:} In the following, we provide upper bounds for~${a_{11}^{i}, a_{31}^{i}}$ and~$a_{41}^{i}$,~${i \in \overline{1:d}}$, as defined in \eqref{terms_a_i}. We start by analyzing the terms ${a^{i}_{11}}$. We use that~${a-b = \tfrac{a^2-b^2}{a+b}}$ to get
    \begin{align*}
        a^{i}_{11}= -\sqrt{\snext_i}\tfrac{({\ett}\gamma'(\xi ))^2s_{i} - \beta^2\snext_{i}}{{\ett}\gamma'(\xi )\sqrt{s_{i}} + \beta\sqrt{\snext_{i}}}.
    \end{align*}
    By using \eqref{sys:s} on $\snext_{i}$ in the numerator, we further obtain
    \begin{align*}
        a^{i}_{11} &= -\sqrt{\snext_i}\tfrac{[({\ett}\gamma'(\xi ))^2 - \beta^2(1-\beta)]s_{i} - \beta^3g_{i}^2}{{\ett}\gamma'(\xi )\sqrt{s_{i}} + \beta\sqrt{\snext_{i}}} \\
        & =-\sqrt{\snext_i}\tfrac{[({\ett}\gamma'(\xi ))^2 - \beta^2(1-\beta)]s_{i}}{{\ett}\gamma'(\xi )\sqrt{s_{i}} + \beta\sqrt{\snext_{i}}}  \\&+ \sqrt{\snext_i}\tfrac{\beta^3g_{i}^2}{{\ett}\gamma'(\xi )\sqrt{s_{i}} + \beta\sqrt{\snext_{i}}}.
    \end{align*}
    We define now~${a^{i}_{12}\coloneqq -\sqrt{\snext_i}\tfrac{({\ett}\gamma'(\xi ))^2- \beta^2(1-\beta)}{{\ett}\gamma'(\xi )\sqrt{s_{i}} + \beta\sqrt{\snext_{i}}}}s_{i}$
    and~${a^{i}_{13}\coloneqq \sqrt{\snext_i} \tfrac{\beta^3 g_i^2}{{\ett}\gamma'(\xi )\sqrt{s_{i}} + \beta\sqrt{\snext_{i}}}}$ to observe that~${a_{11}^{i} = a_{12}^{i} + a_{13}^{i}}$.
    
    We start by upper bounding~${a_{12}^{i}}$. By using algebraic manipulations,
    we get
    \begin{align*}
        a^{i}_{12} = -{\ett}\gamma'(\xi)\tfrac{1 - \tfrac{\beta^2(1-\beta)}{({\ett}\gamma'(\xi))^2}}{\tfrac{\sqrt{s_i}}{\sqrt{\snext_i}} + \tfrac{\beta}{{\ett}\gamma'(\xi)}}s_i.
    \end{align*}
    Note from \eqref{gzero_gbar} that~${\gzero \geq \tfrac{3\beta}{\etzero} > 0}$ and~${\gone > 0}$, and thereby~${\gamma}$, defined in~\eqref{lyapunov:gamma}, is monotonically increasing on~$[0,\infty)$. We now use that~${\dst \geq 0}$ and~${\gzero \geq \tfrac{3\beta}{\etzero}}$ to get
    \begin{align*}
        {\ett}\gamma'(\xi)&\geq \etzero \gzero, \hspace{0.25cm}(*) &  \etzero \gzero &\geq 3\beta, \hspace{0.25cm}(**)
    \end{align*}
    and from \eqref{sys:s} that~${\sqrt{\snext_i} = \sqrt{(1-\beta)s_i + \beta g_i^2} \geq \sqrt{1-\beta}\sqrt{s_i}}$ 
    to obtain
    \begin{align}
        \label{a_12:upper_bound}
        a^{i}_{12} &\overset{(*)}{\leq} -{\ett}\gamma'(\xi)\tfrac{1 - \tfrac{\beta^2(1-\beta)}{(\etzero \gzero)^2}}{\tfrac{\sqrt{s_i}}{\sqrt{(1-\beta)s_i}} + \tfrac{\beta}{\etzero \gzero}}s_i \notag
        \\ &\overset{(**)}{\leq} -{\ett}\gamma'(\xi) \tfrac{8 + \beta}{\tfrac{9}{\sqrt{1-\beta}} + 3}s_i \eqqcolon -{\ett}\gamma'(\xi) c_{\gamma} s_i.
    \end{align}
    
    We proceed by upper bounding~${a_{13}^{i}}$. Observe that, since~${s_i \geq 0}$ and from \eqref{lyapunov:gamma} and \eqref{gzero_gbar} that~${\gamma'(\xi) > 0}$, we have 
    \begin{align}
        \label{step:1:3:1}
        {\ett}\gamma'(\xi )\sqrt{s_{i}} + \beta\sqrt{\snext_{i}} \geq \beta \sqrt{\snext_i}.
    \end{align}
    Additionally, from \eqref{sys:s},~${\snext_i = (1-\beta)s_i + \beta g_i^2 \geq \beta g_i^2}$, which we plug together with \eqref{step:1:3:1} into~${a_{13}^{i}}$ to get
    \begin{align}
        \label{a_13:upper_bound}
        a^{i}_{13} \leq \sqrt{\snext_i}\tfrac{\beta^3g_i^2}{\beta\sqrt{\beta g_i^2}} = \sqrt{\snext_i}\beta^{\frac{3}{2}} \lvert g_i \rvert. 
    \end{align}
     Thus, by using \eqref{a_12:upper_bound} and \eqref{a_13:upper_bound}, we get
     \begin{align}
         \label{a_11:ineq}
         a_{11}^{i} = a_{12}^{i} + a_{13}^{i} \leq -{\ett}\gamma'(\xi) c_{\gamma} s_i + \sqrt{\snext_i}\beta^{\frac{3}{2}} \lvert g_i \rvert.
     \end{align}

    Next, we use $(*)$, $(**)$ and that~${\gamma'}$ is positive on~${[0,\infty)}$. Then, for~${a_{31}^{i}}$, defined in \eqref{terms_a_i}, we get
    \begin{align}
        \label{a_31:upper_bound}
        a^{i}_{31} \overset{(*)}{\leq} -{\ett}\gamma'(\xi )\tfrac{\varepsilon}{2}\left(1 - \tfrac{2\beta}{\etzero \gzero}\right) \overset{(**)}{\leq} -{\ett}\gamma'(\xi)\tfrac{\varepsilon}{6}.
    \end{align}
    For~${a_{41}^{i}}$, defined in \eqref{terms_a_i}, we first add and subtract the term~${\ett\gamma'(\xi)\tfrac{\etone L}{2}}$, and by a subsequent rearrangement, we get
    \begin{align}
        \label{a_41_prelim}
        a^{i}_{41} &= -\ett \gamma'(\xi) \left(\varepsilon - \tfrac{(\etzero + \etone)  L}{2}\right) - \ett\gamma'(\xi)\frac{\ettwo}{\varepsilon} \notag \\ &+ \beta\varepsilon + {\dt}{\ett}\gamma'(\xi) \tfrac{ L}{2},
    \end{align}
    where~${\ettwo\coloneqq \tfrac{\etone\varepsilon L}{2}}$. 
    Since~${\varepsilon - \frac{(\etzero + \etone) L}{2} > 0}$ holds by assumption, i.e., since~${\etzero + \etone < \tfrac{2\varepsilon}{L}}$, we use the inequality~$(*)$ on the first right hand side term in \eqref{a_41_prelim}. We then get 
    \begin{align}
        \label{a_41_prelim:2}
        a^{i}_{41} &\leq -\etzero \gzero \left(\varepsilon - \tfrac{(\etzero + \etone)  L}{2}\right) - \ett\gamma'(\xi)\frac{\ettwo}{\varepsilon} \notag \\ &+ \beta\varepsilon + {\dt}{\ett}\gamma'(\xi) \tfrac{ L}{2}.
    \end{align}
    Moreover,~${\gamma_0 \geq \tfrac{\beta \varepsilon}{\etzero \left(\varepsilon - \tfrac{(\etzero + \etone)  L }{2}\right)}}$ directly follows from \eqref{gzero_gbar}, which we then plug into the right hand side of~\eqref{a_41_prelim:2} to get
    \begin{align}
        \label{a_41:upper_bound}
        a^{i}_{41} \leq - \ett\gamma'(\xi)\frac{\ettwo}{\varepsilon} + {\dt}{\ett}\gamma'(\xi) \tfrac{ L}{2}.
    \end{align}
    By plugging \eqref{a_11:ineq}, \eqref{a_31:upper_bound} and \eqref{a_41:upper_bound} into the right hand side of~\eqref{proof:theorem:K:0}, we get for all~${i \in \overline{1:d}}$ that
    \begin{align}
        \label{proof:theorem:K}
        \kappa_{i}&\leq -{\ett}\gamma'(\xi)c_{\gamma}s_{i} - a^{i}_{21} -{\ett}\gamma'(\xi)\tfrac{\varepsilon}{6}\sqrt{\snext_i}  \notag \\
        &- \ett\gamma'(\xi)\ettwo +{\dt}{\ett}\gamma'(\xi) \tfrac{ \varepsilon L}{2} \notag \\ &+ \sqrt{\snext_i}\left( - {\ett}\gamma'(\xi)\tfrac{\varepsilon}{6} + \beta^{\frac{3}{2}} \lvert g_i \rvert \right).
    \end{align}

    \textbf{Step 1.4:} In the following we analyze the terms
    \begin{align}
        \label{a_51_midway:1}
        a^{i}_{51}\coloneqq - {\ett}\gamma'(\xi)\tfrac{\varepsilon}{6} + \beta^{\frac{3}{2}} \lvert g_i \rvert
    \end{align}
    from \eqref{proof:theorem:K} and show for all~${i \in \overline{1:d}}$ that~${a_{51}^{i} \leq \dt(c_1 + c_2\dt)}$ holds for some positive constants~${c_1}$ and~$c_2$.
    We first require a lower bound for~${\xi}$, defined in \eqref{def:xi:final}, for which an equivalent representation is given by
    \begin{align}
        \label{def:xi:lambda}
        \xi = \lambda f_0(\xnext) + (1-\lambda) f_0(x)
    \end{align}
    for some~${\lambda \in [0,1]}$. Under \Cref{assump}~\ref{assump:convex},~$f$ is (strongly) convex, from where we get~${f_0(\xnext) \geq f_0(x) + \nabla f^{\top}(x)(\xnext - x)}$. By plugging this lower bound into the right hand side of \eqref{def:xi:lambda} and by using that~${s_i \geq 0}$ and~${\varepsilon > 0}$, we get 
    \begin{align}
        \label{xi_midstep:1}
        \xi &\geq f_0(x) - \lambda{\ett}\sum\limits_{i=1}^{d} \tfrac{g_i^2}{\varepsilon + \sqrt{(1-\beta)s_i + \beta g_i^2}} \notag \\ & \geq f_0(x) - \lambda{\ett}\sum\limits_{i=1}^{d} \tfrac{g_i^2}{\sqrt{\beta g_i^2}} \geq f_0(x) - \tfrac{{\ett}}{\sqrt{\beta}}\sum\limits_{i=1}^{d} \lvert g_i \rvert,
    \end{align}
    where in the last inequality we used that $\lambda \leq  1$. In addition,  we use  that~${\mathrm{AM}\geq \mathrm{GM}}$ to get
    \begin{align}
        \label{norm_g_i_upper_bound}
        \lvert g_i \rvert \leq w(\dst) g_i^2 + \tfrac{1}{4w(\dst)}
    \end{align}
    for any~${w(\dst) > 0}$.
    Choose~${w(\dst) \coloneqq \tfrac{\sqrt{\beta}}{4{\ett} L}}$. Moreover, under \Cref{assump}~\ref{assump:convex} and~\ref{assump:L_smooth} we have~${f_0(x) \geq \tfrac{\lVert g \rVert^2}{2 L}}$ (see e.g., 
    \cite[Theorem 2.1.5, (2.1.10)]{nesterov2013introductory}), which we plug together with \eqref{norm_g_i_upper_bound} into the right hand side of \eqref{xi_midstep:1}
    to get 
    \begin{align}
    \label{proof:xi_lower_bound}
        \xi \geq \tfrac{1}{4 L}\lVert g \rVert^2 - \ett^2c_L\eqqcolon \xi_0({\ett}),
    \end{align}
    where~${c_L\coloneqq\tfrac{Ld}{\beta}}$.
    From \eqref{def:xi:final} it follows directly that~${\xi \geq 0}$, thus, we derive the lower bound
    \begin{align}
        \label{xi:lower_bound_prelim}
        \xi \geq \mathrm{max}\big\{0, \xi_0({\ett})\big\}.
    \end{align}
    
    We now analyze the term~${a_{51}^{i}}$ by plugging into \eqref{xi:lower_bound_prelim} in~${\gamma'(\xi) = \gzero + \gone \sqrt{\xi} }$ on the right hand side of \eqref{a_51_midway:1}, i.e.,
    \begin{align}
        \label{a_51:prelim_bound:1}
        a^{i}_{51} \leq -\ett\tfrac{\gzero \varepsilon}{6} - \ett \tfrac{\gone \varepsilon}{6}\sqrt{\mathrm{max}\big\{0, \xi_0({\ett})\big\}} + \beta^{\frac{3}{2}}\lvert g_i\rvert.
    \end{align}
    We observe the two possible outcomes of~${\mathrm{max}\big\{0, \xi_0({\ett})\big\}}$, with~${\xi_0(\dst)}$ as in \eqref{proof:xi_lower_bound}, and therefore define the partition
    \begin{subequations}
    \label{partition:G:1}
        \begin{align}
            G_1&\coloneqq \{x \in \R^d: \|g\|^2 \leq {\ett}^2 4Lc_L\}, \\
            G_2&\coloneqq \{x \in \R^d: \|g\|^2 > {\ett}^2 4Lc_L\}.
        \end{align}
    \end{subequations}
    We now find an upper bound of the right hand side of \eqref{a_51:prelim_bound:1} over both partition sets which holds for all~${i \in \overline{1:d}}$.
    
    First, we analyze the right hand side of \eqref{a_51:prelim_bound:1} 
    for any~${x \in G_1}$. In this case, we have~${\mathrm{max}\big\{0, \xi_0({\ett})\big\} = 0}$, and by using that~${g_i^2 \leq \|g\|^2 \leq \ett^2 4Lc_L}$ and with~${c_L = \tfrac{Ld}{\beta}}$, we get
    \begin{align*}
        a^{i}_{51} \leq {\ett} \left( -\tfrac{\varepsilon}{6}\gzero + 2\beta  L \sqrt{d}\right)\eqqcolon a_{52}.
    \end{align*}
    From \eqref{gzero_gbar} it follows that~${\gzero > \tfrac{12 \beta L \sqrt{d}}{\varepsilon}}$, thus~${a_{52} \leq 0}$. 
    
    Second, we analyze the right hand side of~\eqref{a_51_midway:1} for any~${x \in G_2}$. Then, we have~${\mathrm{max}\big\{0, \xi_0({\ett})\big\} = \xi_0(\dst) > 0}$, and we use the inequality~${\sqrt{a - b} \geq \sqrt{a} - \sqrt{b}}$ on~${\sqrt{\xi_0(\dst)}}$,  which holds for any~${a \geq b \geq 0}$ . Subsequently, by plugging the result into the right hand side of \eqref{a_51:prelim_bound:1} and rearranging the terms, we get
    \begin{align}
    \label{a_51:prelim}
        a^{i}_{51} &\leq {\ett}\tfrac{\varepsilon}{6}\left(-\gamma_0 + {\ett}\gamma_1 \sqrt{c_L}\right) - {\ett}\tfrac{\gone\varepsilon}{12\sqrt{L}}\lVert g\rVert \notag \\ &+ \beta^{\frac{3}{2}}\lvert g_i \rvert.
    \end{align}
    From \eqref{gzero_gbar} one can directly verify that~${\gzero > \etzero\gone\sqrt{c_L}}$. By using that~${\dst \geq 0}$, as well as norm inequalities, we further get~${ {\ett} \lVert g \rVert \geq \etzero  \lvert g_i \rvert }$. By plugging these two lower bounds into the right hand side of~\eqref{a_51:prelim}, we get 
    \begin{align*}
        a^{i}_{51} &\leq {\dt}{\ett}\tfrac{\gamma_1\varepsilon\sqrt{c_L}}{6} - \tfrac{{\etzero}\gone\varepsilon}{12\sqrt{L}}|g_i| + \beta^{\frac{3}{2}}\lvert g_i \rvert.
    \end{align*}
    Finally, by plugging in~${\gone}$, defined in \eqref{gzero_gbar}, in the second right hand side term of the above relation, we get 
    \begin{align}
        \label{a_51_upper bound}
        a_{51}^{i} \leq {\dt}{\ett}\tfrac{\gamma_1\varepsilon\sqrt{c_L}}{6} \eqqcolon {\dt} {\ett} a_{53}.
    \end{align}
    Thus, $a_{51}^{i} \leq \max\{a_{52}, {\dt} {\ett} a_{53}\} = {\dt} {\ett} a_{53}$, which we then plug into the last right hand side term of \eqref{proof:theorem:K} to get
    \begin{align}
        \label{proof:theorem:K:next}
        \kappa_{i}&\leq -{\ett}\gamma'(\xi)c_{\gamma}s_{i} - a^{i}_{21} -{\ett}\gamma'(\xi)\tfrac{\varepsilon}{6}\sqrt{\snext_i} \notag \\
        &- \ett\gamma'(\xi)\ettwo +\dst{\ett}\gamma'(\xi) \tfrac{\varepsilon L}{2}  + \dt\ett a_{53}\sqrt{\snext_i}.
    \end{align}

    \textbf{Step 1.5:}
    In the following, we finalize the upper bounds for~$\kappa_i$,~${i \in \overline{1:d}}$, in \eqref{proof:theorem:K:next}. We first find an upper and a lower bound for all $\snext_i$,~${i \in \overline{1:d}}$, and then plug them into the right hand side of \eqref{proof:theorem:K:next}. Starting from \eqref{sys:s} and by using that~${\mathrm{QM} \geq \mathrm{AM}}$, i.e.,~${-\mathrm{QM} \leq -\mathrm{AM}}$, we have 
    \begin{align}
        \label{step:1:5:1}
        &-\sqrt{\snext_i} = -\sqrt{2}\tfrac{1}{\sqrt{2}}\sqrt{(1-\beta) s_i + \beta g_i^2} \notag \\ &\leq -\tfrac{\sqrt{2(1-\beta)}}{2}\sqrt{s_i} - \tfrac{\sqrt{2\beta}}{2} \lvert g_i \rvert \eqqcolon - a_{61}\sqrt{s_i} - a_{62} \lvert g_i \rvert,
    \end{align}
    and further, by using the sub-additivity property of~${\sqrt{\cdot}}$, i.e.,~${\sqrt{a + b} \leq \sqrt{a} + \sqrt{b}}$ for any~${a,b \geq 0}$, we get
    \begin{align}
    \label{step:1:5:2}
        \sqrt{\snext_i} &= \sqrt{(1-\beta) s_i + \beta g_i^2} \leq \sqrt{1-\beta}\sqrt{s_i} + \sqrt{\beta}\lvert g_i \rvert.
    \end{align}
    Hence, by plugging \eqref{step:1:5:1} and \eqref{step:1:5:2} into the third and sixth term on the right hand side of \eqref{proof:theorem:K:next}, respectively, we get
    \begin{align*}
        \kappa_{i}&\leq -{\ett}\gamma'(\xi)c_{\gamma}s_{i} - a^{i}_{21} \notag \\ & -{\ett}\gamma'(\xi)\tfrac{\varepsilon}{6} \left(a_{61}\sqrt{s_i} + a_{62} \lvert g_i \rvert\right) - \ett\gamma'(\xi)\ettwo \notag \\ &+\!{\dt}{\ett}\gamma'(\xi) \tfrac{\varepsilon L}{2} \!+\! {\dt} {\ett} a_{53}\sqrt{1-\beta}\sqrt{s_i}  \notag \\ &+ {\dt} {\ett} a_{53}\sqrt{\beta} \lvert g_i \rvert.
    \end{align*}
    By using algebraic manipulations and plugging in~${a_{21}^i}$ from~\eqref{terms_a_i}, we further get
    \begin{align}
    \label{proof:theorem:K:12}
        \kappa_{i}&\leq {\ett}\gamma'(\xi)\Big(-c_{\gamma}s_{i} - \sqrt{s_i}\left(\varepsilon - \tfrac{\etzero  L}{2}\right) -\tfrac{\varepsilon}{6}a_{61}\sqrt{s_i} \notag \notag \\ &- a_7 \lvert g_i \rvert - \ettwo +{\dt}(\sqrt{s_i} + \varepsilon) \tfrac{ L}{2}\notag  \\ &+ {\dt}  a_{53}\sqrt{1-\beta}\tfrac{\sqrt{s_i}}{\gamma'(\xi)} + {\dt} a_{53}\sqrt{\beta} \tfrac{\lvert g_i \rvert}{\gamma'(\xi)}\Big),
    \end{align}
    where~${a_7\coloneqq \tfrac{\varepsilon a_{62}}{6}}$.
    Trivially, it holds that~${\tfrac{\sqrt{s_i}}{\gamma'(\xi)} \leq \tfrac{\sqrt{s_i}}{\gzero}}$ due to \eqref{lyapunov:gamma} and \eqref{gzero_gbar}, and~${\xi \geq 0}$. 
    
    In the following, we upper bound~${ \tfrac{|g_i|}{\gamma'(\xi)}}$ for all~${|g_i| \geq 0}$. By using \eqref{lyapunov:gamma}, ${\gzero, \gone > 0}$ from \eqref{gzero_gbar} and~${\xi_0(\dst)}$ from~\eqref{xi:lower_bound_prelim}, we get
    \begin{align}
    \label{ratio_g_xi:1}
        \tfrac{\lvert g_i \rvert}{\gamma'(\xi)} &= \tfrac{\lvert g_i \rvert}{\gzero + \gamma_1\sqrt{\xi}} \leq \tfrac{\lvert g_i \rvert}{\gzero + \gamma_1\sqrt{\mathrm{max}\{0, \xi_0({\ett})\}}} 
        \notag \\ & \leq \tfrac{\lvert g_i \rvert}{\gzero + \gamma_1\sqrt{\mathrm{max}\Big\{0, \tfrac{1}{4L}g_i^2 - {\ett}^2c_L \Big\}}},
    \end{align}
    where, to obtain the right hand side upper bound, we used that~${g_i^2 \leq \lVert g \rVert^2}$, and where~${c_L=\tfrac{Ld}{\beta}}$.
    We observe the two possible outcomes of~${\mathrm{max}\Big\{0, \tfrac{1}{4L}g_i^2 - \ett^2c_L\Big\}}$, and define the following partitions for all~${i \in \overline{1:d}}$:
    \begin{subequations}
    \label{partition:G:2}
        \begin{align}
            G_3^i&\coloneqq \{x \in \R^d: g_i^2 \leq {\ett}^2 4Lc_L\} \\
            G_4^{i}&\coloneqq \{x \in \R^d: g_i^2 > {\ett}^2 4Lc_L\}.
        \end{align}
    \end{subequations}
    We now find a supremum of the right hand side of \eqref{ratio_g_xi:1} over the partition sets which holds for all~${i \in \overline{1:d}}$.
    
    We first consider any~${x \in G_3^i}$. We then have~${g_i^2  \leq {\ett}^24Lc_L}$, and thereby for \eqref{ratio_g_xi:1} we have
    \begin{align}
        \label{max_ratio_1}
        \tfrac{\lvert g_i \rvert}{\gamma'(\xi)} \leq \ett\tfrac{2\sqrt{Lc_L}}{\gzero} \eqqcolon\lin_1(\dst).
    \end{align}
    
    Second, we consider any~${x \in G_4^i}$. We then have~${g_i^2  > {\ett}^24Lc_L}$. We use the bound~${\sqrt{a-b} \geq \sqrt{a} - \sqrt{b}}$ for any~${a \geq b \geq 0}$, and plug~${{\ett} = \etzero  + {\dt}}$ into the right hand side of~\eqref{ratio_g_xi:1}
    to get
    \begin{align*}
        \tfrac{\lvert g_i \rvert}{\gamma'(\xi)} \leq \tfrac{\lvert g_i \rvert}{\gamma_0 - \gone\etzero \sqrt{c_L} + \tfrac{\gone}{2\sqrt{L}}\lvert g_i \rvert - {\dt}\gone \sqrt{c_L}}.
    \end{align*}
    From \eqref{gzero_gbar} we have ${\gzero \geq \gamma_1\etzero \sqrt{c_L} + \tfrac{12 \beta L \sqrt{d}}{\varepsilon}}$, thus
    \begin{align}
    \label{ratio_max}
        \tfrac{\lvert g_i \rvert}{\gamma'(\xi)} \leq \tfrac{\lvert g_i \rvert}{\tfrac{12 \beta L \sqrt{d}}{\varepsilon} + \tfrac{\gone}{2\sqrt{L}}\lvert g_i \rvert - \dst\gone \sqrt{c_L}}.
    \end{align}

    We now analyze the right hand side of~\eqref{ratio_max} as a function of~${|g_i|}$, and find an upper bound which holds for all~${ g_i^2  > \ett^24Lc_L}$, i.e., for all~${x \in G_{4}^i}$.
    First, if~${\tfrac{12 \beta L \sqrt{d}}{\varepsilon} - \dst\gone \sqrt{c_L} \geq 0}$, then the right hand side of~\eqref{ratio_max} is upper bounded by its supremum, which occurs for~${|g_i| \to \infty}$. Second, if~${\tfrac{12 \beta L \sqrt{d}}{\varepsilon} - \dst\gone \sqrt{c_L}  < 0}$, then the right hand side of~\eqref{ratio_max} is upper bounded by the supremum at~${g_i^2 = \ett^24Lc_L}$.
    Thus, we have 
    \begin{align}
        \label{max_ratio_2}
        \tfrac{\lvert g_i \rvert}{\gamma'(\xi)} &\leq \mathrm{max}\Big\{\tfrac{2 \sqrt{L}}{\gamma_1}, \ett c_m \Big\} \leq \tfrac{2 \sqrt{L}}{\gamma_1} + \ett c_m \eqqcolon \lin_2(\dst),
    \end{align}
    where ${c_m:=\tfrac{2\varepsilon\sqrt{Lc_L}}{12\beta\sqrt{d}+ \etzero\gone \varepsilon \sqrt{c_L}}}$.
    
    Then, we get from \eqref{max_ratio_1} and \eqref{max_ratio_2} for all~${x \in \R^d}$ that~${\tfrac{\lvert g_i \rvert}{\gamma'(\xi)} \leq \mathrm{max}\big\{\lin_1(\dst), \lin_2(\dst)\big\} \leq \lin_1(\dst) +\lin_2(\dst)\eqqcolon\lin_3({\dt})}$, and further recall that~${\tfrac{\sqrt{s_i}}{\gamma'(\xi)} \leq \tfrac{\sqrt{s_i}}{\gzero}}$. By plugging both upper bounds into the right hand side of~\eqref{proof:theorem:K:12} and rearranging the terms, we get
    \begin{subequations}
    \begin{align}
        \label{proof:theorem:K:2}
        \kappa_{i} \leq {\ett}& \gamma'(\xi)\big( {\dt} \lin_5({\dt}) \!- \!c_{\gamma}s_{i}\! + \!\lin_4({\dt}) \sqrt{s_i} \! - \! \ettwo \!- \! a_7 \lvert g_i \rvert \big), \\
        \label{nu_delta}
        \lin_4({\dt}) &\coloneqq  \tfrac{\etzero  L}{2} -\varepsilon -\tfrac{\varepsilon}{6}a_{61} + {\dt}\tfrac{ L}{2} + {\dt} a_{53}\tfrac{\sqrt{1-\beta}}{\gzero}, \\
        \label{phi_delta}
        \lin_5({\dt}) &\coloneqq \tfrac{\varepsilon L}{2} + a_{53}\sqrt{\beta}\lin_3({\dt}).
    \end{align}
    \end{subequations}

    \textbf{Part 2:} In this part, as outlined at the beginning of the proof, we derive a suitable upper bound for \eqref{proof:delta_V_ineq} which depends solely on~${\gmax}$,~${\smax}$ and~${\dst \in \R_{\geq 0}}$.

    \textbf{Step 2.1:} In the following, we derive an upper and a lower bound for $\xi$, given in \eqref{def:xi:lambda}, which are dependent only on~${\gmax}$ and~${\dst}$. From \eqref{def:xi:final} it follows that~${\xi \geq 0}$, and from norm inequalities that~${\lVert g \rVert^2 \geq \gmax^2}$. Thus, we lower bound the right hand side of  \eqref{xi:lower_bound_prelim} as follows:
    \begin{align}
        \label{xi_g_m_lower_bound}
        \xi & \geq \mathrm{max}\Big\{0, \tfrac{1}{4L}\gmax^2 - \ett^2c_L\Big\}.
    \end{align}
    
    Next, we find an upper bound for~\eqref{def:xi:lambda}. Under \Cref{assump}~\ref{assump:L_smooth}, we have 
    \begin{align*}
        f_0(\xnext) - f_0(x) \leq \nabla f_0^{\top}(x)(\xnext - x) + \tfrac{L}{2} \lVert \xnext - x \rVert_2^2,
    \end{align*}
    which we rewrite as in \eqref{proof:theorem:L_smoothness_ineq} and subsequently plug it into the right hand side of~\eqref{def:xi:lambda}. We then have
    \begin{align*}
        \xi \leq f_0(x) - \lambda {\ett} \sum\limits_{i=1}^{d} \tfrac{g_i^2}{\varepsilon + \sqrt{\snext_i}}  + \tfrac{\lambda {\ett}^2 L}{2} \sum\limits_{i=1}^{d} \tfrac{g_i^2}{(\varepsilon + \sqrt{\snext_i})^2}.
    \end{align*}
    We use~${0 \leq \lambda \leq 1}$, and, following from \eqref{sys:s} and~${\varepsilon > 0}$, that~${\varepsilon + \sqrt{\snext_i} \geq \sqrt{\beta g_i^2}}$ to get
    \begin{align*}
        \xi \leq f_0(x) + \ett^2\tfrac{ L}{2} \sum\limits_{i=1}^{d}\tfrac{g_i^2}{\beta g_i^2} \leq f_0(x) + \ett^2\tfrac{ c_L}{2},
    \end{align*}
    with~${c_L=\tfrac{Ld}{\beta}}$.
    Finally, under \Cref{assump}~\ref{assump:convex} with~${\mu > 0}$ it holds that~${2\mu f_0(x) \leq \lVert g \rVert^2}$ (see, e.g., 
    \cite[Theorem 2.1.10, (2.1.24)]{nesterov2013introductory}), and by further using~${\lVert g \rVert^2 \leq d \G^2}$ we get
    \begin{align}
    \label{xi_g_m_upper_bound}
        \xi \leq \tfrac{d }{2\mu}\G^2 + \ett^2\tfrac{ c_L}{2}.
    \end{align}

    \textbf{Step 2.2:} In the following, we provide two upper bounds for $\tfrac{g_{i}^2\kappa_i}{(\varepsilon + \sqrt{\snext_i})^2(\varepsilon + \sqrt{s_i})}$, $i \in \overline{1:d}$, as given in~\eqref{proof:delta_V_ineq}, which depend solely on~${\gmax}$,~${\smax}$ and~$\dst$. Observe first in \eqref{proof:theorem:K:2} that the upper bound of~$\kappa_i$ contains the term~${\ett\gamma'(\xi)\left(-c_\gamma s_i + \lin_{4}({\dt}) \sqrt{s_i} + {\dt} \lin_{5}({\dt})\right)}$. 
    By analyzing the left hand side of
    \begin{align*}
        -c_\gamma s_i + \lin_{4}({\dt}) \sqrt{s_i} + {\dt} \lin_{5}({\dt}) \leq 0
    \end{align*}
    as a quadratic polynomial in~$\sqrt{s_i}$, we get that the left hand side expression attains a maximum for~$\sqrt{s_i} = \tfrac{\lin_{4}({\dt})}{2c_\gamma}$. Moreover, we then get for all~${s_i> \left(\tfrac{\lin_{4}({\dt})}{2c_\gamma} + \tfrac{\sqrt{\lin_{4}^2({\dt}) +{\dt}\lin_{5}({\dt})4c_\gamma}}{2c_\gamma}\right)^2}$ that~$\kappa_i$ is negative, and this holds for all~${\lvert g_i \rvert \geq 0}$. If $\lin_{4}({\dt}) \geq 0$, this upper bound is realized as~${s_i \geq 0}$. Otherwise, the maximum is attained for~${s_i = 0}$. We thus obtain for all~${i \in \overline{1:d}}$ that
    \begin{align*}
        \kappa_i &\leq {\ett} \gamma'(\xi) \left(\tfrac{\lin_4^2(\dst)\mathrm{max}\{0, \mathrm{sgn}(\lin_4(\dst))\}}{4c_\gamma} + {\dt} \lin_{5}({\dt})\right) \\ &\eqqcolon{\ett}\gamma'(\xi) \hat{\Gamma}_1({\dt}),
    \end{align*}
    where we used that, trivially,~${-\ettwo < 0}$.
    Moreover, by using that~${s_i \geq 0}$,~${\snext_i \geq 0}$ and~${\varepsilon > 0}$, as well as~${\snext_i \geq  \beta g_i^2}$ from~\eqref{sys:s}, we get
    \begin{align}
        \label{upper_bound_positive_kappa}
        \tfrac{g_{i}^2\kappa_i}{\left(\varepsilon + \sqrt{\snext_i}\right)^2(\varepsilon + \sqrt{s_i})} &\leq \tfrac{g_{i}^2{\ett}\gamma'(\xi) \hat{\Gamma}_1({\dt})}{ \varepsilon \beta g_i^2 } = \ett \gamma'(\xi)\hat{\Gamma}_2(\dst),
    \end{align}
    where~${\hat{\Gamma}_2(\dst)\coloneqq \tfrac{\hat{\Gamma}_1(\dst)}{\varepsilon\beta}}$.
    From \eqref{nu_delta} and \eqref{phi_delta} it can be directly verified that~${\hat{\Gamma}_2}$ is continuous and positive definite on~${[0,\infty)}$. By further checking  \eqref{max_ratio_1} and \eqref{max_ratio_2} it can be verified that the terms in~${\hat{\Gamma}}_2$ are either monotonically increasing  or strictly monotonically increasing on~${[0,\infty)}$. Thus, we have that~${\hat{\Gamma}_2}$ is strictly monotonically increasing on~$[0,\infty)$.
    
    With \eqref{upper_bound_positive_kappa}, we established our first upper bound. We now provide a second upper bound for~${\tfrac{g_{i}^2\kappa_i}{(\varepsilon + \sqrt{\snext_i})^2(\varepsilon + \sqrt{s_i})}}$,~${i \in \overline{1:d}}$. By rearranging the right hand side of~\eqref{proof:theorem:K:2}, we have 
    \begin{align*}
        \kappa_i &\leq -{\ett}\gamma'(\xi)\left(c_{\gamma}s_{i} + a_7 \lvert g_i \rvert + \ettwo \right) \\
        & + {\ett}\gamma'(\xi)\left(\lin_{4}({\dt}) \sqrt{s_i} + {\dt} \lin_{5}({\dt})\right).
    \end{align*}
    Multiplying both sides with $\tfrac{g_i^2}{(\varepsilon + \sqrt{\snext_i})^2(\varepsilon + \sqrt{s_i})} \geq 0$, we get
    \begin{align}
        \label{step22:1}
        \tfrac{g_{i}^2\kappa_i}{\left(\varepsilon + \sqrt{\snext_i}\right)^2(\varepsilon + \sqrt{s_i})} &\leq -\tfrac{g_{i}^2{\ett}\gamma'(\xi)\left(c_{\gamma}s_{i} + a_7 \lvert g_i \rvert + \ettwo \right)}{\left(\varepsilon + \sqrt{\snext_i}\right)^2(\varepsilon + \sqrt{s_i})} \notag \\
        & +\tfrac{g_{i}^2{\ett}\gamma'(\xi)\left(\lin_{4}({\dt}) \sqrt{s_i} + {\dt} \lin_{5}({\dt})\right)}{\left(\varepsilon + \sqrt{\snext_i}\right)^2(\varepsilon + \sqrt{s_i})}.
    \end{align}
    We now find an upper bound for the second right hand side term of \eqref{step22:1}. By using~${s_i \geq 0}$ and, from \eqref{phi_delta}, that~${\lin_5}$ is positive on~${[0,\infty)}$, 
    we get~${\tfrac{\lin_{4}({\dt}) \sqrt{s_i}}{\varepsilon + \sqrt{s_i}} \leq \mathrm{max}\big\{0, \lin_{4}(\dst)\big\}}$ and~${\tfrac{{\dt} \lin_{5}({\dt})}{\varepsilon + \sqrt{s_i}} \leq \tfrac{\dst \lin_{5}(\dst)}{\varepsilon}}$. Moreover, from \eqref{sys:s} we get~${\snext_i \geq \beta g_i^2}$ which we plug into~${\tfrac{\ett \gamma'(\xi)g_i^2}{(\varepsilon + \sqrt{\snext_i})^2}}$ along with~${\varepsilon > 0}$ to get
    \begin{align*}
        \tfrac{\ett \gamma'(\xi)g_i^2}{\left(\varepsilon + \sqrt{\snext_i}\right)^2} \leq \tfrac{\ett \gamma'(\xi)g_i^2}{\left(\varepsilon + \sqrt{\beta g_i^2}\right)^2} \leq \ett \gamma'(\xi)\tfrac{g_i^2}{\beta g_i^2}  = \ett\frac{ \gamma'(\xi)}{\beta}.
    \end{align*}
    By plugging the above inequalities into the second right hand side term of \eqref{step22:1}, we get
    \begin{align}
        \label{upper_bound_positive_kappa:2}
        \tfrac{g_{i}^2\kappa_i}{\left(\varepsilon + \sqrt{\snext_i}\right)^2(\varepsilon + \sqrt{s_i})} &\leq \ett\gamma'(\xi)\tfrac{\mathrm{max}\{0, \lin_{4}(\dst)\}}{\beta} \notag \\ &+ \dst \ett\gamma'(\xi)\tfrac{\lin_{5}(\dst)}{\varepsilon \beta} \notag \\ &  -{\ett}\gamma'(\xi)\tfrac{g_{i}^2\left(c_{\gamma}s_{i} + a_7 \lvert g_i \rvert + \ettwo\right)}{\left(\varepsilon + \sqrt{\snext_i}\right)^2(\varepsilon + \sqrt{s_i})}.
    \end{align}

    \textbf{Step 2.3:} In the following we derive an upper bound for the right hand side of \eqref{proof:delta_V_ineq} dependent only on~${\gmax}$,~${\smax}$ and~${\dst}$ by using the results from Step 2.1 and Step 2.2. 
    
    We start by analyzing the second right hand side sum of~\eqref{proof:delta_V_ineq}, and rearranging it as 
    \begin{align*}
        \sum\limits_{i=1}^{d}\tfrac{g_{i}^2\kappa_i}{\left(\varepsilon + \sqrt{\snext_i}\right)^2(\varepsilon + \sqrt{s_i})} &= \tfrac{g_{m}^2\kappa_m}{\left(\varepsilon + \sqrt{\snext_m}\right)^2(\varepsilon + \sqrt{s_m})} \\ & + \sum\limits_{\substack{j=1 \\ j \neq m}}^{d} \tfrac{g_{j}^2\kappa_j}{\left(\varepsilon + \sqrt{\snext_j}\right)^2(\varepsilon + \sqrt{s_j})},
    \end{align*}
    where~${m \in \underset{i \in \overline{1:d}}{\mathrm{argmax}} |g_i|}$. In the above equation, for each term in the sum on the right hand side we use the upper bound in \eqref{upper_bound_positive_kappa}, and for the first right hand side term (i.e., the $m$-th component) we use the upper bound in \eqref{upper_bound_positive_kappa:2}. Then, we get
    \begin{align}
        \label{step23:1}
        &\sum\limits_{i=1}^{d}\tfrac{g_{i}^2\kappa_i}{\left(\varepsilon + \sqrt{\snext_i}\right)^2(\varepsilon + \sqrt{s_i})} \leq  \ett\gamma'(\xi)\tfrac{\mathrm{max}\{0, \lin_{4}(\dst)\}}{\beta} \notag  \\ 
        +& \dst \ett\gamma'(\xi)\tfrac{\lin_{5}(\dst)}{\varepsilon \beta} - {\ett}\gamma'(\xi)\tfrac{\gmax^2\left(c_{\gamma}s_{m} + a_7  \gmax + \ettwo \right)}{\left(\varepsilon + \sqrt{\snext_m}\right)^2(\varepsilon + \sqrt{s_m})}
       \notag  \\ +&  \ett\gamma'(\xi)\Gamma(\dst),
    \end{align}
    where we used that~${|g_m| = \gmax}$ and where~${\Gamma(\dst)\coloneqq (d-1)\hat{\Gamma}_2({\dt})}$.

    Next, we analyze the first right hand side sum of \eqref{proof:delta_V_ineq}, which we remember is negative definite in~${s \in \R_{\geq 0}^d}$. It therefore trivially holds that
    \begin{align}
    \label{step23:2}
        -\beta\sum\limits_{i=1}^{d}\tfrac{s_i}{\varepsilon + \sqrt{s_i}} \leq - \beta \tfrac{\smax}{\varepsilon + \sqrt{\smax}}.
    \end{align}
    We now provide an upper bound for~${\Delta \Vv}$. By plugging \eqref{step23:1} and \eqref{step23:2} into the right hand side of~\eqref{proof:delta_V_ineq}, we get
    \begin{align}
        \label{delta_V_step_2_3:1}
        \Delta \Vv &\leq - \beta \tfrac{\smax}{\varepsilon + \sqrt{\smax}} + \nlin_1(\dst)\gamma'(\xi) \notag \\
        &- {\ett}\gamma'(\xi)\tfrac{\gmax^2\left(c_{\gamma}s_{m} + a_7  \gmax + \ettwo  \right)}{\left(\varepsilon + \sqrt{\snext_m}\right)^2(\varepsilon + \sqrt{s_m})},
    \end{align}
    where
    \begin{align}
        \label{def:p_1}
        \nlin_1(\dst) &\coloneqq \ett \Gamma(\dst) + \ett\tfrac{\mathrm{max}\{0, \lin_{4}(\dst)\}}{\beta}  + \dst\ett\tfrac{ \lin_{5}(\dst)}{\varepsilon \beta}.
    \end{align}
    We note that~${\nlin_1}$ is positive definite, continuous and strictly monotonically increasing on~${[0,\infty)}$. To obtain the desired upper bound of \eqref{proof:delta_V_ineq}, we also use that
    \begin{align*}
        \snext_m &= (1-\beta)s_m + \beta \G^2 \leq (1-\beta)\smax + \beta \G^2, 
    \end{align*}
    and that~${\ett \geq \etzero}$, due to~${u \geq 0}$, and subsequently plug both inequalities into the third right hand side term of \eqref{delta_V_step_2_3:1}. Then, we get
    \begin{align}
        \label{delta_V_midway_step23}
        \Delta \Vv &\leq -\beta \tfrac{\smax}{\varepsilon + \sqrt{\smax}}  + \nlin_1(\dst)\gamma'(\xi) \notag \\
        &  - \etzero \gamma'(\xi) \tfrac{ \G^2\left( c_\gamma s_m + a_7 \G + \ettwo\right)}{\left(\varepsilon + \sqrt{(1-\beta)\smax + \beta \G^2}\right)^2(\varepsilon + \sqrt{s_m})}.
    \end{align}

    In the following we eliminate the dependence of \eqref{delta_V_midway_step23} on~${\xi}$. We first focus on the term~${\nlin_{1}(\dst)\gamma'(\xi)}$ in \eqref{delta_V_midway_step23}. We use from \eqref{lyapunov:gamma} that~${\gamma'(\xi) = \gzero + \gone\sqrt{\xi}}$ and plug in the upper bound on~${\xi}$ from \eqref{xi_g_m_upper_bound}, as well as use the sub-additivity property of~${\sqrt{\cdot}}$. Then, 
    \begin{align}
        \label{xi_uppp}
        \gamma'(\xi) = \gzero+ \gone\sqrt{\xi} \leq \gzero + \hat{\gamma}_1\gmax + \ett\hat{\gamma}_2,
    \end{align}
    with~${\hat{\gamma}_1\coloneqq \tfrac{\gone\sqrt{d}}{\sqrt{2\mu}}}$ and~${\hat{\gamma}_2\coloneqq\tfrac{\gone \sqrt{c_L}}{\sqrt{2}}}$.
    
    Next, we focus on~${\gamma'(\xi)}$ in the third right hand side term of \eqref{delta_V_midway_step23} and use the lower bound on~${\xi}$ from \eqref{xi_g_m_lower_bound}. Then, 
    \begin{align}
        \label{def:rho_1}
        \gamma'(\xi) &\geq \rho_1(\gmax, \dst) \notag \\ &\coloneqq \gzero + \gone\sqrt{\mathrm{max}\big\{0, \tfrac{1}{4L}\gmax^2 - \ett^2c_L\big\}}.
    \end{align}
    Thereafter, by plugging \eqref{xi_uppp} and \eqref{def:rho_1} into the second and third right hand side terms of \eqref{delta_V_midway_step23}, respectively, we get
    \begin{align}
        \label{delta_V_midway_step23:2}
        \Delta \Vv &\leq -\beta \tfrac{\smax}{\varepsilon + \sqrt{\smax}} \notag \\& + \nlin_1(\dst)\left(\gzero + \hat{\gamma}_1\gmax + \ett\hat{\gamma}_2\right) \notag \\
        &  -  \etzero\tfrac{ \rho_1(\gmax, \dst) \G^2\left( c_\gamma s_m + a_7 \G + \ettwo \right)}{\left(\varepsilon + \sqrt{(1-\beta)\smax + \beta \G^2}\right)^2(\varepsilon + \sqrt{s_m})}.
    \end{align}

    Finally, we leave it to the reader to verify that
    \begin{align}
    \label{case:2:inequality}
        -\tfrac{c_\gamma s_m + a_7 \G + \ettwo}{\varepsilon + \sqrt{s_m}} &\leq \underset{s_m \in \R_{\geq 0}}{\mathrm{max}}\left(-\tfrac{c_\gamma s_m + a_7  \G + \ettwo}{\varepsilon + \sqrt{s_m}}\right)
        \notag \\ &= -2c_{\gamma}\left(\sqrt{\varepsilon^2 + \tfrac{a_7}{c_{\gamma}}\gmax + \tfrac{\ettwo}{c_\gamma}} - \varepsilon\right) \notag \\ &\eqqcolon-\rho_2(\gmax),
    \end{align}   
    where the right hand side's maximum is attained at~${s_m = \left(\sqrt{\varepsilon^2 + \tfrac{a_{7}}{c_{\gamma}} \gmax + \tfrac{\ettwo}{c_\gamma}} - \varepsilon\right)^2}$. Then, for \eqref{delta_V_midway_step23:2} we get the desired upper bound
    \begin{align}
        \label{delta_V_ISS_ineq}
        \Delta \Vv &\leq -\beta \tfrac{\smax}{\varepsilon + \sqrt{\smax}} + \nlin_2(\dst) + \nlin_3(\dst)\gmax \notag \\
        &  -  \etzero\tfrac{ \rho_1(\gmax, \dst) \G^2\rho_2(\gmax)}{\left(\varepsilon + \sqrt{(1-\beta)\smax + \beta \G^2}\right)^2},
    \end{align}
    where
    \begin{subequations}
    \label{nlin_2_3}
    \begin{align}
        \nlin_2(\dst)&\coloneqq \nlin_1(\dst)\left(\gzero + \ett\hat{\gamma}_2\right), \\
        \nlin_3(\dst)&\coloneqq \nlin_1(\dst)\hat{\gamma}_1.
    \end{align}
    \end{subequations}

    \textbf{Part 3:} In the following, we analyze the right hand side of~\eqref{delta_V_ISS_ineq} for a bipartition of the state space, defined as
    \begin{subequations}
    \label{partition}
        \begin{align}
            S_1 &\coloneqq \{(x,s) \in \R^d \times \R^{d}_{\geq 0} : \smax \geq (L {\dmax})^{2+2q}\}, \\
            S_2 &\coloneqq \{(x,s) \in \R^d \times \R^{d}_{\geq 0} : \smax < (L {\dmax})^{2+2q}\},
        \end{align}
    \end{subequations}
    with~${q \in \left(0, \tfrac{1}{4}\right)}$ and~${y \coloneqq x - x^{*}}$,
    and where ${S_1 \cup S_2 = \R^d \times \R^{d}_{\geq 0}}$.
    For each region we find upper bounds for~${\Delta \Vv}$, from which we further infer that for all~${(x,s) \in \R^d \times \R_{\geq 0}^{d}}$ and any~${\dst \in \R_{\geq 0}}$, we have
    \begin{align}
    \label{ISS:ideal:last}
        \Delta \Vv \leq -\alpha_{\Vv}(\lVert (y, s) \rVert_{\infty}) + \chi_{\Vv}(\dst),
    \end{align}
    where $\alpha_{\Vv}, \chi_{\Vv} \in \mathcal{K}_{\infty}$.

    \textbf{Step 3.1: Region $S_1$:} 
    Let ${(x,s) \in S_1}$ as defined in \eqref{partition}. Starting from \eqref{delta_V_ISS_ineq}, in the following we derive an upper bound for~${\Delta \Vv}$ for all~${(x,s) \in S_1}$. 
    
    Under \Cref{assump} \ref{assump:L_smooth}, by using norm inequalities one can show that~${L {\dmax} \geq \G}$. Then, for all~${(x,s) \in S_1}$ we have~${\smax \geq \left(L {\dmax}\right)^{2 + 2q} \geq \G^{2+2q} }$. Next, we split the first right hand side term of \eqref{delta_V_ISS_ineq} as
    \begin{align*}
        -\tfrac{\smax}{\varepsilon + \sqrt{\smax}} = -\sum\limits_{j=1}^{3} \tfrac{1}{3} \tfrac{\smax}{\varepsilon + \sqrt{\smax}},
    \end{align*}
    and observe that the right hand side functions are indeed monotonically decreasing on~$[0,\infty)$. By using~${\smax \geq \G^{2 + 2q}}$, we get 
    \begin{align*}
        -\tfrac{1}{3} \tfrac{\smax}{\varepsilon + \sqrt{\smax}} \leq -\tfrac{1}{3} \tfrac{\gmax^{2+2q}}{\varepsilon + \gmax^{1+q}},
    \end{align*}
    and moreover,
    \begin{align}
        \label{case:1:0}
        -\tfrac{\smax}{\varepsilon + \sqrt{\smax}} \leq -\tfrac{1}{3} \tfrac{\smax}{\varepsilon + \sqrt{\smax}} -\tfrac{2}{3} \tfrac{\gmax^{2+2q}}{\varepsilon + \gmax^{1+q}}.
    \end{align}
    Now, observe that the last right hand side term of \eqref{delta_V_ISS_ineq} is negative definite, and thereby upper bounded by zero. Further, we plug \eqref{case:1:0} into the first right hand side term of \eqref{delta_V_ISS_ineq}, and thus get
    \begin{align}
        \label{case_1:1}
        \Delta \Vv &\leq - \tfrac{\beta}{3}\tfrac{\smax}{\varepsilon + \sqrt{\smax}}-\tfrac{2\beta}{3}\tfrac{\G^{2 + 2q}}{\varepsilon + \G^{1+q}}  \notag \\ & + \nlin_2(\dst) + \nlin_3(\dst)\gmax.
    \end{align}
    
    Next, we focus on~${\nlin_3(\dst)\gmax}$. Consider the expression
    \begin{align}
    \label{case_1:vanishing ineq}
        \rho_3(\gmax,\dst)\coloneqq -\tfrac{\beta}{3}\tfrac{\G^{2 + 2q}}{\varepsilon + \G^{1+q}} + \nlin_3(\dst)\G.
    \end{align}
    Then, \eqref{case_1:1} takes the form
    \begin{align}
        \label{case_1:1:secundum}
        \Delta \Vv &\leq - \tfrac{\beta}{3}\tfrac{\smax}{\varepsilon + \sqrt{\smax}}-\tfrac{\beta}{3}\tfrac{\G^{2 + 2q}}{\varepsilon + \G^{1+q}}  \notag \\ & + \nlin_2(\dst) + \rho_3(\gmax,\dst).
    \end{align}
    The negative term in \eqref{case_1:vanishing ineq} grows at a rate of~${\G^{1+q}}$, which dominates the term linear in~$\G$ if~$\gmax$ is sufficiently large with respect to~$\dst$. 
    We now show that one can find some~${v_{1}(\dst) \geq 0}$ such that for all~${ \G \geq v_{1}(\dst)}$ it holds that~${\rho_3(\gmax, \dst) \leq 0}$, 
    and  for all~${\G \geq 0}$ that~${\rho_3(\gmax,\dst) \leq \nlin_3(\dst)v_{1}(\dst)}$. We set
    \begin{align}
        \label{v_1}
        v_1(\dst)\coloneqq\mathrm{max} \left\{\varepsilon^{\tfrac{1}{1+q}}, \left(\tfrac{6 \nlin_3(\dst) }{\beta}\right)^{\tfrac{1}{q}}\right\},
    \end{align}
    which is positive and continuous on~$[0,\infty)$.
    
    In the following, we show that for all~${\gmax \geq v_1(\dst)}$ it holds that~${\rho_3(\gmax,\dst) \leq 0}$. As we assume that~${\gmax \geq v_1(\dst)}$, from \eqref{v_1} we further get~${ \varepsilon \leq \gmax^{1+q}}$. By plugging this inequality into the denominator of the negative definite term on the right hand side of~\eqref{case_1:vanishing ineq}, we get
    \begin{align*}
        -\tfrac{\beta}{3} \tfrac{\gmax^{2+2q}}{\varepsilon + \gmax^{1+q}} \leq -\tfrac{\beta}{6} \gmax^{1+q}.
    \end{align*}
    By plugging the obtained inequality into the right hand side of~\eqref{case_1:vanishing ineq}, we get
    \begin{align}
        \label{case:1:pre_vanishing_final}
        \rho_3(\gmax,\dst) \leq -\tfrac{\beta}{6} \gmax^{1+q} + \nlin_3(\dst)\gmax.
    \end{align}
    Since by assumption we have~${\gmax \geq v_1(\dst)}$, from \eqref{v_1} we get~${\gmax \geq \left(\tfrac{6\nlin_3(\dst)}{\beta}\right)^{\tfrac{1}{q}}}$, i.e.,~${\nlin_3(\dst) \leq \left(\frac{\beta}{6}\gmax\right)^q}$. By plugging this inequality into the right hand side of \eqref{case:1:pre_vanishing_final}, we get
    \begin{align}
        \label{rho_3_ineq}
        \rho_3(\gmax,\dst) \leq 0.
    \end{align}
    
    Finally, we show for all~${\G \geq 0}$ that~${\rho_3(\gmax,\dst) \leq \nlin_3(\dst)v_{1}(\dst)}$. Remember that~${\rho_3(\cdot, u)}$ is continuous on~${[0, v_1(u))}$, with the first term in \eqref{case_1:vanishing ineq} being negative definite, and thus upper bounded by zero, and the second term monotonically increasing. Therefore, for~${\gmax \in [0, v_1(u))}$ we get that
    \begin{align*}
        \rho_3(\gmax, \dst) \leq \nlin_3(\dst)v_1(\dst),
    \end{align*}
    and, trivially, due to the result obtained for all~${\gmax \geq v_1(\dst)}$ in \eqref{rho_3_ineq}, we get for all~${\gmax \geq 0}$ that
    \begin{align*}
        \rho_3(\gmax, \dst) \leq \mathrm{max}\big\{\nlin_3(\dst)v_1(\dst),0\big\}=\nlin_3(\dst)v_1(\dst),
    \end{align*}
    which we then plug into the right hand side of \eqref{case_1:1:secundum} to get
    \begin{align}
    \label{case:1:0:2}
        \Delta \Vv &\leq - \tfrac{\beta}{3}\tfrac{\smax}{\varepsilon + \sqrt{\smax}}-\tfrac{\beta}{3}\tfrac{\G^{2 + 2q}}{\varepsilon + \G^{1+q}}  \notag \\ & + \nlin_2(\dst) + \nlin_3(\dst)v_1(\dst).
    \end{align}

    It only remains to provide an upper bound for~\eqref{case:1:0:2} dependent on~${\xmax}$,~${\smax}$ and~$\dst$. We use here that~${-\tfrac{\beta}{3}\tfrac{\gmax^{2+2q}}{\varepsilon + \gmax^{1+q}}}$ is monotonically decreasing in~${\gmax}$. Under \Cref{assump}~\ref{assump:convex}, by using \cite[Theorem 2.1.10, (2.1.26)]{nesterov2013introductory} and norm inequalities, we have that the inequality~${{\gmax}^2 \geq \tfrac{\mu^2 \xmax^2}{d}}$ applies, thereby obtaining
    \begin{align*}
        -\tfrac{\beta}{3}\tfrac{\gmax^{2+2q}}{\varepsilon + \gmax^{1+q}} \leq -\tfrac{\beta}{3}\tfrac{\hat{r}_1^{2+2q}\xmax^{2 + 2q}}{\varepsilon + \hat{r}_1^{1+q}\xmax^{1+q}},
    \end{align*}
    where~${\hat{r}_1\coloneqq \tfrac{\mu}{\sqrt{d}}}$.
    By plugging the above inequality into the right hand side of \eqref{case:1:0:2}, we get 
    \begin{align}
        \label{case_1:1:2}
        \Delta \Vv &\leq - \tfrac{\beta}{3}\tfrac{\smax}{\varepsilon + \sqrt{\smax}} \notag \\ &-\tfrac{\beta}{3}\tfrac{\hat{r}_1^{2+2q}\xmax^{2 + 2q}}{\varepsilon + \hat{r}_1^{1+q}\xmax^{1+q}}  \notag \\ &+ \nlin_2(\dst) + \nlin_3(\dst)v_{1}(\dst) \notag \\
         &\eqqcolon -\psi_{11}(\smax) - \psi_{12}(\xmax) + \chi_{1}(\dst).
    \end{align}
    
    \textbf{Step 3.2: Region $S_2$:}  Let ${(x,s) \in S_2}$, with~$S_2$ as defined in \eqref{partition}. Starting from~\eqref{delta_V_ISS_ineq}, in the following we derive an upper bound for~${\Delta \Vv}$ for all~${(x,s) \in S_2}$. 
    
    First, observe for all~${(x,s) \in S_2}$ that~${\smax < (L {\dmax})^{2+2q}}$. Second, under \Cref{assump} \ref{assump:convex} with~${\mu > 0}$, by using \cite[Theorem 2.1.10, (2.1.26)]{nesterov2013introductory} and norm inequalities, it follows that~${\G^2 \geq \hat{r}_1^2 \dmax}^2$ holds, with~${\hat{r}_1= \tfrac{\mu}{\sqrt{d}}}$. Combining these two inequalities gives us~${\smax < \hat{r}_2 \G^{2+ 2q}}$ for all~${(x,s) \in S_2}$, where~${\hat{r}_2\coloneqq\left(\tfrac{L^2 d}{\mu^2}\right)^{1 + q}}$, which we plug into the denominator of the last right hand side term of  \eqref{delta_V_ISS_ineq} to get
    \begin{align}
        \label{case:2}
        \Delta \Vv &\leq -\beta \tfrac{\smax}{\varepsilon + \sqrt{\smax}} + \nlin_2(\dst) + \nlin_3(\dst)\gmax \notag \\ &- 2\tfrac{\etzero}{2}\tfrac{  \rho_1(\G, \dst) \G^{2} \rho_2(\gmax)}{\left(\varepsilon + \sqrt{\hat{r}_3\G^{2+ 2q} + \beta \G^2}\right)^2},
    \end{align}
    where~${\hat{r}_3\coloneqq (1-\beta)\hat{r}_2}$.
    
    Next, we focus on~${\nlin_3(\dst)\gmax}$. Consider the expression
    \begin{align}
    \label{case_2:vanishing ineq}
        \rho_4(\gmax, \dst) \hspace{-0.025cm}&\coloneqq \hspace{-0.075cm}- \tfrac{\etzero}{2}\tfrac{  \rho_1(\G, \dst) \G^{2} \rho_2(\gmax)}{\left(\varepsilon + \sqrt{\hat{r}_3\G^{2+ 2q} + \beta \G^2}\right)^2} \hspace{-0.075cm} + \hspace{-0.075cm}\nlin_3(\dst)\gmax.
    \end{align}
    Then, \eqref{case:2} takes the form
    \begin{align}
        \label{case:2:secundum}
        \Delta \Vv &\leq -\beta \tfrac{\smax}{\varepsilon + \sqrt{\smax}} + \nlin_2(\dst) + \rho_4(\gmax, \dst) \notag \\ &- \tfrac{\etzero}{2}\tfrac{  \rho_1(\G, \dst) \G^{2} \rho_2(\gmax)}{\left(\varepsilon + \sqrt{\hat{r}_3\G^{2+ 2q} + \beta \G^2}\right)^2},
    \end{align}
    The negative definite term on the right hand side of~\eqref{case_2:vanishing ineq}, for~${\gmax}$ sufficiently large with respect to~$\dst$, grows with a rate of~${\G^{1 + 2 + \tfrac{1}{2} - 2 - 2q } = \gmax^{1 + \tfrac{1}{2} - 2q}}$. The reader can verify this by direct analysis of~\eqref{def:rho_1} and~\eqref{case:2:inequality}. As~${2q < \tfrac{1}{2}}$ holds due to the choice of~$q$ in \eqref{partition}, this implies that $\gmax^{1 + \tfrac{1}{2} - 2q}$ grows with a rate greater than~$\G$ and thus, for a sufficiently large~$\G$ dominates the term linear in~$\G$.  Similarly to the analysis done for Region $S_1$, we now show that there exists some~${v_2(\dst) \geq 0}$ such that for all~${\G\geq  v_2(\dst)}$ it holds that~${\rho_4(\gmax,\dst) \leq 0}$, 
    and moreover, for all~${\G \geq 0}$ that~${\rho_4(\gmax,\dst) \leq  \nlin_3(\dst)v_2(\dst)}$. We set
    \begin{align}
        \label{v_2}
        v_2(\dst) &\coloneqq \mathrm{max}\Bigg\{\ett 2\sqrt{2Lc_L}, \tfrac{\varepsilon^2 c_{\gamma}}{a_7}, \left(\tfrac{\varepsilon + \epsilon_2}{\sqrt{\hat{r}_3 + \epsilon_1 }}\right)^{\tfrac{1}{1+q}}, \notag \\
        & \left(\nlin_3(\dst)\tfrac{8(\sqrt{2}+1)\sqrt{2L}(\hat{r}_3 + \epsilon_1)}{\etzero\gone\sqrt{a_7c_\gamma}}\right)^{\tfrac{2}{1-4q}}\Bigg\},
    \end{align}
    which is positive and continuous on~$[0,\infty)$, and where~${\epsilon_1}$ and~${\epsilon_2}$ are positive constants such that for all~${\gmax \geq 0}$
    \begin{align}
        \label{epsilon_1_epsilon_2}
        \beta \gmax^2 \leq \epsilon_1 \gmax^{2+2q} + \epsilon_2^2.
    \end{align}
    One such pair of constants are, for example,~${\epsilon_1 = \beta}$ and~${\epsilon_2 = \sqrt{q \beta (1+q)^{-\frac{q+1}{q}}}}$.

    In the following, we show that for all~${\gmax \geq v_2(\dst)}$ we have~${\rho_4(\gmax,\dst) \leq 0}$. 
    We start by lower bounding~${\rho_1}$, defined in \eqref{def:rho_1}. Since \eqref{gzero_gbar} implies~${\gzero > 0}$, we get 
    \begin{align}
    \label{step_2_4_reg_2_0}
        \rho_1(\gmax,\dst) &> \gamma_1\sqrt{\mathrm{max}\Big\{0, \tfrac{1}{4L}\gmax^2 - \ett^2c_L\Big\}}.
    \end{align}
    Now, by assumption we have~${\gmax\geq v_2(\dst)}$.  From~\eqref{v_2} it then follows that~${\gmax \geq \ett 2\sqrt{2Lc_L}}$, from where we get~${\ett^2 c_L \leq \tfrac{\gmax^2}{8L}}$. By plugging this inequality into the right hand side of \eqref{step_2_4_reg_2_0}, we get
    \begin{align}
        \label{step_2_4_reg_2_1}
        \rho_1(\gmax,\dst) > \gamma_1\sqrt{\tfrac{\gmax^2}{4L} - \tfrac{\gmax^2}{8L}} = \tfrac{\gamma_1}{2\sqrt{2L}}\gmax.
    \end{align}
    
    Next, we upper bound the denominator in the first term on the right hand side of~\eqref{case_2:vanishing ineq}. By plugging~\eqref{epsilon_1_epsilon_2} into the said denominator, we get
    \begin{align*}
        \sqrt{\hat{r}_3 \gmax^{2+2q} + \beta \gmax^2} \leq \sqrt{(\hat{r}_3 + \epsilon_1) \gmax^{2+2q} + \epsilon_2^2}.
    \end{align*}
    By using the sub-additivity property of~$\sqrt{\cdot}$, we further get
    \begin{align}
        \label{denominator:1}
        \sqrt{(\hat{r}_3 + \epsilon_1) \gmax^{2+2q} + \epsilon_2^2} &\leq \sqrt{\hat{r}_3 + \epsilon_1}\gmax^{1+q} + \epsilon_2.
    \end{align}
    Next, by assumption, we have~${\gmax\geq v_2(\dst)}$, and from \eqref{v_2} we further get~${\gmax \geq \left(\tfrac{\varepsilon + \epsilon_2}{\sqrt{\hat{r}_3 + \epsilon_1 }}\right)^{\tfrac{1}{1+q}}}$. Then, we obtain that
    \begin{align}
        \label{denominator:2}
        \varepsilon + \epsilon_2 \leq \sqrt{\hat{r}_3 + \epsilon_1}\gmax^{1+q}.
    \end{align}
    By using \eqref{denominator:1} and \eqref{denominator:2}, for the denominator of the first term on the right hand side of~\eqref{case_2:vanishing ineq} we get
    \begin{align}
    \label{step_2_4_reg_2_2}
        &\left(\varepsilon + \sqrt{\hat{r}_3 \gmax^{2+2q} + \beta \gmax^2}\right)^2 \notag
        \\ \leq &\left(\varepsilon + \epsilon_2 + \sqrt{\hat{r}_3 + \epsilon_1 }\gmax^{1+q}\right)^2 \leq 4\left(\hat{r}_3 + \epsilon_1\right)\gmax^{2+2q}.
    \end{align}

    Further, we lower bound~${\rho_2}$, defined in \eqref{case:2:inequality}. We first observe that the inequality
    \begin{align*}
        \rho_2(\gmax) \geq 2c_\gamma\left(\sqrt{\varepsilon^2 + \tfrac{a_7}{c_\gamma}\gmax} - \varepsilon\right)
    \end{align*}
    holds since~${\frac{\ettwo}{c_\gamma} > 0}$, and subsequently use that~${a-b = \tfrac{a^2 - b^2}{a+b}}$ to get
    \begin{align}
        \label{partition:2:rho_2}
        \rho_2(\gmax) \geq \tfrac{ 2a_7\gmax }{\sqrt{\varepsilon^2 + \tfrac{a_7}{c_\gamma}\gmax} + \varepsilon}.
    \end{align}
    By assumption, we have~${\gmax\geq v_2(\dst)}$, with~$v_2(\dst)$ defined in \eqref{v_2}, which implies~${\gmax \geq \tfrac{\varepsilon^2 c_\gamma}{a_7}}$. Thus, we get~${\varepsilon^2 \leq \tfrac{a_7 \gmax}{c_{\gamma}}}$, which we plug into the right hand side of~\eqref{partition:2:rho_2} to get 
    \begin{align}
    \label{step_2_4_reg_2_3}
        \rho_2(\gmax) \geq  \tfrac{2a_7\gmax}{(\sqrt{2}+1)\sqrt{\tfrac{a_7}{c_\gamma} \gmax}} 
        = \tfrac{2\sqrt{a_7 c_\gamma}}{\sqrt{2}+1}\sqrt{\gmax}.
    \end{align}

    Finally, by plugging in \eqref{step_2_4_reg_2_1}, \eqref{step_2_4_reg_2_2} and \eqref{step_2_4_reg_2_3} into the right hand side of \eqref{case_2:vanishing ineq} and by using elementary algebraic inequalities, we get
    \begin{align*}
        &\rho_4(\gmax,\dst) \notag \\ \leq& \gmax\left(-\frac{\etzero\gone\sqrt{a_7c_\gamma}\gmax^{\tfrac{1}{2} - 2q}}{8(\sqrt{2}+1)\sqrt{2L}(\hat{r}_3 + \epsilon_1)} + \nlin_3(\dst)\right).
    \end{align*}
    By assumption, we have~${\gmax\geq v_2(\dst)}$, with~$v_2(\dst)$ defined in \eqref{v_2}, which implies~${\gmax \geq \left(\nlin_3(\dst)\tfrac{8(\sqrt{2}+1)\sqrt{2L}(\hat{r}_3 + \epsilon_1)}{\etzero\gone\sqrt{a_7c_\gamma}}\right)^{\tfrac{2}{1-4q}}}$. By plugging this into the above inequality, we get
    \begin{align}
        \label{rho_4_ineq}
        \rho_4(\gmax,\dst) \leq 0.
    \end{align}

    Finally, we show for all~${\G \geq 0}$ that~${\rho_4(\gmax,\dst) \leq \nlin_3(\dst)v_{2}(\dst)}$. Remember that~${\rho_4(\cdot, u)}$ is continuous on~${[0, v_2(u))}$, with the first right hand side term of \eqref{case_1:vanishing ineq} being negative definite, and thus upper bounded by zero, and the second term is monotonically increasing. Therefore, for~${\gmax \in [0, v_2(u))}$ we get that
    \begin{align*}
        \rho_4(\gmax, \dst) \leq \nlin_3(\dst)v_2(\dst),
    \end{align*}
    and trivially, due to the result obtained for all~${\gmax \geq v_2(\dst)}$ in \eqref{rho_4_ineq}, for all~${\gmax \geq 0}$ we get
    \begin{align*}
        \rho_4(\gmax, \dst) \leq \mathrm{max}\big\{\nlin_3(\dst)v_2(\dst),0\big\}=\nlin_3(\dst)v_2(\dst).
    \end{align*}
    By plugging this inequality into the right hand side of \eqref{case:2:secundum}, we get
    \begin{align}
        \label{case:2:1}
        \Delta \Vv &\leq -\beta \tfrac{\smax}{\varepsilon + \sqrt{\smax}} + \nlin_2(\dst) + \nlin_3(\dst)v_2(\dst) \notag \\ &- \tfrac{\etzero}{2}\tfrac{  \rho_1(\G, \dst) \G^{2} \rho_2(\gmax)}{\left(\varepsilon + \sqrt{\hat{r}_3\G^{2+ 2q} + \beta \G^2}\right)^2}.
    \end{align}
    
    It only remains to provide an upper bound for~\eqref{case:2:1} dependent on~${\xmax}$,~${\smax}$ and~$\dst$. We first observe from~\eqref{gzero_gbar} and~\eqref{def:rho_1} that~${\rho_1(\gmax, \dst) \geq \gzero > 0}$ for all~${\gmax \geq 0}$ and~${\dst \in \R_{\geq 0}}$. Second, we observe that~${\rho_2}$, defined in~\eqref{case:2:inequality}, is a monotonically increasing function on~$[0,\infty)$. We use the inequality~${ \tfrac{\mu^2 \xmax^2}{d} \leq {\gmax}^2}$, which holds due to \Cref{assump}~\ref{assump:convex} with~${\mu > 0}$. Then, for the numerator of the last term of \eqref{case:2:1} we get
    \begin{align}
        \label{case:2:1:1}
          \rho_1(\G, \dst) \G^{2} \rho_2(\gmax) \geq  \gzero \hat{r}_1^2\xmax^{2} \rho_2\left(\hat{r}_1\xmax\right),
    \end{align}
    where~${\hat{r}_1 = \tfrac{\mu}{\sqrt{d}}}$.
    We next use that~${\gmax^2 \leq L^2d \xmax^2}$ which holds under \Cref{assump}~\ref{assump:L_smooth}, and plug it in the denominator of the last right hand side term of \eqref{case:2:1} to get
    \begin{align}
        \label{case:2:1:2}
        \hat{r}_3\G^{2+ 2q} + \beta \G^2 &\leq \hat{r}_4\xmax^{2+2q} +\hat{r}_5\xmax^2,
    \end{align}
    where~${\hat{r}_4\coloneqq \hat{r}_3(L^2d)^{1+q}}$ and~${\hat{r}_5\coloneqq \beta L^2d}$. 
    By plugging \eqref{case:2:1:1} and \eqref{case:2:1:2} into the last term of \eqref{case:2:1}, we get
    \begin{align}
        \label{case:2:2}
        \Delta \Vv &\leq - \beta\tfrac{\smax}{\varepsilon + \sqrt{\smax}} \notag \\ &- \tfrac{\etzero\gzero\hat{r}_1^2}{2}\tfrac{   \xmax^{2} \rho_2\left(\hat{r}_1\xmax\right)}{\left(\varepsilon + \sqrt{\hat{r}_4\xmax^{2+2q} +\hat{r}_5\xmax^2}\right)^2} \notag \\ & + \nlin_2(\dst) + \nlin_3(\dst)v_2(\dst) \notag \\
         &\eqqcolon -\psi_{21}(\smax) - \psi_{22}(\xmax) + \chi_{2}(\dst).
    \end{align}

    \textbf{Step 3.3: Full State Space $S_1 \cup S_2$:} 
    In the following, we prove the claim of \Cref{ISS-theorem}. First, we show global asymptotic stability of algorithm \eqref{sys:rmsprop} according to \cite[Theorem 3.4.6.]{michel2008stability} for~${u = 0}$ and with respect to the equilibrium~${(x^{*},0)}$. Second, we show ISS of algorithm \eqref{sys:rmsprop}  for any~${\dst(\cdot) \in \Uu}$
    with respect to the equilibrium~$(x^{*},0)$ according to \cite[Definition 3.2.]{jiang2001iss_discrete}, adapted for the considered setup, by analyzing~${\Delta \Vv}$ for any~${\dst \in \R_{\geq 0}}$.

    We start by finding an upper bound for~${\Delta \Vv}$ for all~${(x,s) \in \rrd}$ based on the the upper bounds for~${\Delta \Vv}$ from steps 3.1 and 3.2. Observe that, for~${(x,s) \in S_j}$, ${j \in \{1,2\}}$, as given in \eqref{case_1:1:2} and \eqref{case:2:2}, we get an upper bound of the type
    \begin{align*}
        \Delta \Vv \leq -\psi_{j1}(\smax) - \psi_{j2}(\xmax) + \chi_j(\dst).
    \end{align*}

   It can be verified that the functions~${\psi_{11}}$,~${\psi_{12}}$ and~${\psi_{21}}$ are continuous, positive definite, strictly monotonically increasing and radially unbounded on~${[0, \infty)}$, thus, they belong to class~${\mathcal{K}_{\infty}}$. Moreover, it can be verified that~${\psi_{22}}$ is continuous, positive definite and radially unbounded on~${[0, \infty)}$. The latter holds true as the numerator in~${\psi_{22}(\xmax)}$ grows at a rate of~${\xmax^{\tfrac{5}{2}}}$ and the denominator at a rate of~${\xmax^{2+2q}}$ with~${q \in \left(0, \tfrac{1}{4}\right)}$. Then, from \cite[Lemma 4.3]{khalil2002nonlinear} it follows that there exists some~${\psi_{23} \in \mathcal{K}_{\infty}}$, such that for all~${z \in \R_{\geq 0}}$ it holds that~${\psi_{23}(z) \leq \psi_{22}(z)}$. We remark here that the proof of~\cite[Lemma 4.3]{khalil2002nonlinear} applies directly for~${D = \R_{\geq 0}}$ by substituting~${\|x\|}$ with~${x}$. One such candidate is
   \begin{align*}
       \psi_{23}(z): = \begin{cases}
           - \tfrac{\etzero\gzero\hat{r}_1^2c_\gamma}{4\varepsilon^2} z^2 \left(\sqrt{\varepsilon^2 + \tfrac{a_7}{c_\gamma}\hat{r}_1 z} -  \varepsilon\right), & z \leq z_{*} \\
           - \tfrac{\etzero\gzero\hat{r}_1^2c_\gamma}{4} \tfrac{\sqrt{\varepsilon^2 + \tfrac{a_7}{c_\gamma}\hat{r}_1 z} -  \varepsilon}{\hat{r}_4z^{2q} + \hat{r}_5}, & z > z_{*}
       \end{cases} ,
   \end{align*}
    where~${z_{*} \in \R_{\geq 0}}$ is the unique real root of~${\varepsilon = z\sqrt{\hat{r}_4 z^{2q} + \hat{r}_5}}$. The lower bound~${\psi_{23}}$ is obtained by analyzing the denominator of~${\psi_{22}}$. In particular, for any~${z \geq 0}$ such that~${\varepsilon \geq z\sqrt{\hat{r}_4z^{2q} + \hat{r}_5}}$, we upper bound the denominator by~${4\varepsilon^2}$, while for any~${z \geq 0}$ such that~${\varepsilon < z\sqrt{\hat{r}_4z^{2q} + \hat{r}_5}}$, we upper bound the denominator by~${4z^2(\hat{r}_4z^{2q} + \hat{r}_5)}$.

    We now use that, for any~${\psi_3, \psi_4: \R_{\geq 0} \to \R_{\geq 0}}$ and any~${a,b \in \R_{\geq 0}}$, it holds that
    \begin{align*}
        \psi_3(a) + \psi_4(b) \geq \mathrm{min}\left(\psi_3(\|(a,b)\|_{\infty}), \psi_4(\|(a,b)\|_{\infty})\right).
    \end{align*}
    This property can be shown by a simple case analysis and is left for the reader to check.
    Then, we get that
    \begin{align*}
        \psi_{11}(a) + \psi_{12}(b) &\geq \mathrm{min}\left(\psi_{11}(\|(a,b)\|_{\infty}), \psi_{12}(\|(a,b)\|_{\infty})\right) \\
        &\eqqcolon\alpha_{1}(\|(a,b)\|_{\infty}),\\
        \psi_{21}(a) + \psi_{22}(b) &\geq \psi_{21}(a) + \psi_{23}(b)\\
        &\geq \mathrm{min}\left(\psi_{21}(\|(a,b)\|_{\infty}), \psi_{23}(\|(a,b)\|_{\infty})\right) \\
        &\eqqcolon \alpha_2(\|(a,b)\|_{\infty}).
    \end{align*}
    Substituting~${a = \smax}$ and~${b = \xmax}$, and observing that~${\|(\xmax, \smax) \|_{\infty} = \| (y,s) \|_{\infty}}$, we get for all~${(x,s) \in S_j}$, ${j \in \{1,2\}}$ that
    \begin{align*}
        \Delta \Vv &\leq   -\alpha_j(\| (y,s) \|_{\infty}) +  \chi_j(\dst),
    \end{align*}
    where~${\alpha_j \in \mathcal{K}_{\infty}}$,~${j \in \{1,2\}}$, as the minimum of~${\mathcal{K}_{\infty}}$-functions belongs to class~${\mathcal{K}_{\infty}}$. Further, for all~${(x,s) \in \rrd }$ it holds that
    \begin{align}
        \label{DVV}
        \Delta \Vv &\leq  -\underset{i \in \{1,2\}}{\mathrm{min}} \hspace{0.1cm}\alpha_i(\| (y,s) \|_{\infty}) +  \underset{j \in \{1,2\}}{\mathrm{max}}\chi_j(\dst) \notag \\
        & \eqqcolon -\alpha_{\Vv}(\| (y,s) \|_{\infty}) +  \chi_{\Vv}(\dst),
    \end{align}
     where~${\alpha_{\Vv} \in \mathcal{K}_{\infty}}$. This holds true as the minimum of~${\mathcal{K}_{\infty}}$-functions belongs to class~${\mathcal{K}_{\infty}}$. 
     For the sake of completeness, we now present the exact form of~${\alpha_{\Vv}}$. For any~${z \geq 0}$, it reads as
    \begin{align}
        \alpha_{\Vv}(z) &\coloneqq  \mathrm{min}\big\{ \psi_5(z), \tfrac{\tau_7  z }{\tau_8 + \sqrt{ z }} 
        \big\}, \\
        \psi_5(z) &\coloneqq  {\small \begin{cases}
            \tau_1 z^2 \left(\sqrt{\varepsilon^2 + \tau_2 z} -  \varepsilon\right), & z \leq z_{*}, \\
           \tau_3 \tfrac{\sqrt{\varepsilon^2 + \tau_4 z} -  \varepsilon}{\tau_5z^{2q} + \tau_6}, & z > z_{*},
           \end{cases}  } \notag
    \end{align}
    where~${z_{*} \in \R_{\geq 0}}$ is the unique real root of~${\varepsilon = z\sqrt{\tau_5 z^{2q} + \tau_6}}$ and
    \begin{subequations}
        \begin{align}
            \tau_1 &\coloneqq \tfrac{\etzero\gzero\hat{r}_1^2c_\gamma}{4\varepsilon^2}, &
            \tau_2 &\coloneqq \tfrac{a_7 \hat{r}_1}{c_\gamma}, \\
            \tau_3 &\coloneqq \tfrac{\etzero\gzero\hat{r}_1^2c_\gamma}{4}, &
            \tau_4 &\coloneqq \tfrac{a_7 \hat{r}_1}{c_\gamma}, \\
            \tau_5 &\coloneqq \hat{r}_4 & \tau_6 & \coloneqq \hat{r}_5, \\
            \tau_7 &\coloneqq \tfrac{\beta}{3}\mathrm{min}\{1, \hat{r}_1^{1+q}\}, & \tau_8 & \coloneqq \varepsilon \mathrm{max}\{1, \hat{r}_1^{-1-q}\}.
        \end{align}
    \end{subequations}
    Further, by expanding~${\chi_{\Vv}}$ by using \eqref{case_1:1:2} and \eqref{case:2:2},
    we get
    \begin{align*}
        \chi_{\Vv}(\dst) = \nlin_2(\dst) + \nlin_3(\dst)\underset{j \in \{1,2\}}{\mathrm{max}} v_j(\dst),
    \end{align*}
    with~$\nlin_2$ and~$\nlin_3$ as defined in \eqref{nlin_2_3}. It can be verified that~$\nlin_2$ and~$\nlin_3$ are continuous, positive definite, strictly monotonically increasing and radially unbounded on~$[0,\infty)$. In addition, both~$v_1$ and~$v_2$, defined in \eqref{v_1} and \eqref{v_2}, respectively, are positive, continuous and monotonically increasing on~$[0,\infty)$. Therefore,
    we have that~${\chi_{\Vv}\in \mathcal{K}_{\infty}}$.

    In the following, we establish the stability and ISS of algorithm \eqref{sys:rmsprop}. 
    First, we show global asymptotic stability of the zero-input system~(${\dst = 0}$) with respect to the equilibrium~${(x^{*},0)}$, which lies on the boundary of~${\rrd}$. We apply \cite[Theorem 3.4.6.]{michel2008stability} with~${X\coloneqq\rrd}$,~${A\coloneqq\rrd}$,~${M\coloneqq\{(x^{*},0)\}}$ and the distance metric~${d}$ induced by the infinity norm. Under this setup, \Cref{lemma:Lyapunov_fcn} satisfies condition~(3.64) in \cite[Theorem 3.4.6.]{michel2008stability}. Moreover, for~${\dst = 0}$, for \eqref{DVV} we have 
    \begin{align*}
        \Delta \Vv \leq -\alpha_{\Vv}(\| (y,s) \|_{\infty})
    \end{align*}
    as~${\chi_{\Vv}(0) = 0}$. Thus, condition~(3.65) in \cite[Theorem 3.4.6.]{michel2008stability} is satisfied, and thereby, it follows that algorithm \eqref{sys:rmsprop} is globally asymptotically stable with respect to the equilibrium~${(x^{*},0)}$.
    
    Next, we consider any~${\dst \in \R_{\geq 0}}$
    . We invoke 
    \cite[Definition 3.2.]{jiang2001iss_discrete} adapted to the state space~${\rrd}$, equilibrium~${(x^{*},0)}$ for the zero-input system, and the distance metric induced by the infinity norm. From \Cref{lemma:Lyapunov_fcn} and \eqref{DVV} it directly follows that the conditions $(5)$ and~$(6)$ in the adapted ISS definition in \cite[Definition 3.2.]{jiang2001iss_discrete} are satisfied for all~${(x,s) \in \rrd}$ and all~${\dst \in \R_{\geq 0}}$. Consequently,~${\Delta \Vv}$, defined in \eqref{proof:theorem:difference_Lyapunov}, satisfies an ISS decrease condition, and moreover,~${\Vv}$, defined in \eqref{function:Lyapunov}, is an ISS-Lyapunov function for algorithm \eqref{sys:rmsprop}. Then, by \cite[Lemma 3.5.]{jiang2001iss_discrete}, algorithm \eqref{sys:rmsprop} is ISS for every~${\dst(\cdot) \in \Uu}$ with respect to the equilibrium~${(x^{*},0)}$.
\end{proof}

We now provide some implications and discussions about the main result of this paper.

\begin{rmk}
The stability of algorithm \eqref{sys:rmsprop} can be also approached by a cascade-like argument.
Observe that the $s$-subsystem is driven by the $x$-subsystem via the input $\nabla f(x)$.
Conversely, the $x$-subsystem is driven by the $s$-subsystem and this direction of coupling has a stabilizing effect. 
Therefore, cascade-like stability arguments appear to be applicable for studying the stability of RMSProp.
In particular, for a continuous-time version of RMSProp (see, e.g., \cite{ma2022qualitative}), such an approach seems rather straightforward.
Nevertheless, the availability of an ISS-Lyapunov function is conceptually appealing and offers some benefits, for example,
when dealing with related classes of algorithms.
\end{rmk}

\begin{rmk}
\label{remark:2:trade_off}
    The presented ISS-Lyapunov function in \eqref{function:Lyapunov} can be used to discuss the potential trade-off between higher convergence rates and larger steady-state errors. In particular, the inequality \eqref{DVV} establishes ISS
    in terms of the (gain-like) functions    $\alpha_{\Vv}$ and $\chi_{\Vv}$.
    Notably,~$\alpha_\Vv$ in~\eqref{DVV} is independent of the input~$u$. However, it is possible to get sharper inequalities.
    In particular, for all~${(x,s) \in S_1}$ we have
    \begin{align}
        \label{remark:2:1}
        \Delta \Vv &\leq - \tfrac{\beta}{3}\tfrac{\smax}{\varepsilon + \sqrt{\smax}} -\tfrac{\ett}{2}\tfrac{  \gzero \hat{r}_1^2\xmax^2 \rho_2\left(\hat{r}_1\xmax\right)}{\left(\varepsilon + \sqrt{(1-\beta)\smax +\hat{r}_5\xmax^2}\right)^2} \notag  \\ &  -\tfrac{\beta}{3}\tfrac{\hat{r}_1^{2+2q}\xmax^{2 + 2q}}{\varepsilon + \hat{r}_1^{1+q}\xmax^{1+q}} + \nlin_2(\dst) + \nlin_3(\dst)v_{1}(\dst),
    \end{align}
    and for all~${(x,s) \in S_2}$ we have
    \begin{align}
        \label{remark:2:2}
        \Delta \Vv &\leq - \beta\tfrac{\smax}{\varepsilon + \sqrt{\smax}} - \tfrac{\ett}{2}\tfrac{  \gzero \hat{r}_1^2\xmax^2 \rho_2\left(\hat{r}_1\xmax\right)}{\left(\varepsilon + \sqrt{\hat{r}_4\xmax^{2+2q} +\hat{r}_5\xmax^2}\right)^2}\notag \\ & + \nlin_2(\dst) + \nlin_3(\dst)v_2(\dst),
    \end{align}
    with~$S_1$ and $S_2$ as defined in \eqref{partition}, ${S_1 \cup S_2=\R^d \times \R^d_{\ge 0}}$, and where~${y=x-x^{*}}$. We first highlight here that~${\nlin_2, \nlin_3 \in \mathcal{K}_{\infty}}$. Moreover, the second terms on the right hand side of \eqref{remark:2:1} and~\eqref{remark:2:2} are negative definite in~${y \in \R^d}$, and they are scaled by~$\ett$. Therefore, if~${\ett}$ increases, the convergence speed of algorithm~\eqref{sys:rmsprop} does not decrease and potentially could increase. 
    On the other hand, the term $\nlin_2(\dst) + \nlin_3(\dst)v_2(\dst)$ also increases and induces a larger error floor. Hence, one observes a trade-off between convergence speed and convergence accuracy. Intuitively, the convergence speed appears to increase with~$\ett$, however, a rigorous analysis of this effect is beyond  the scope of this paper and is left for future work. To illustrate this, we provide simulation results in \Cref{fig:1}.
\end{rmk}

\begin{figure}[!t]
    \centering
    \includegraphics[width=1.05\linewidth]{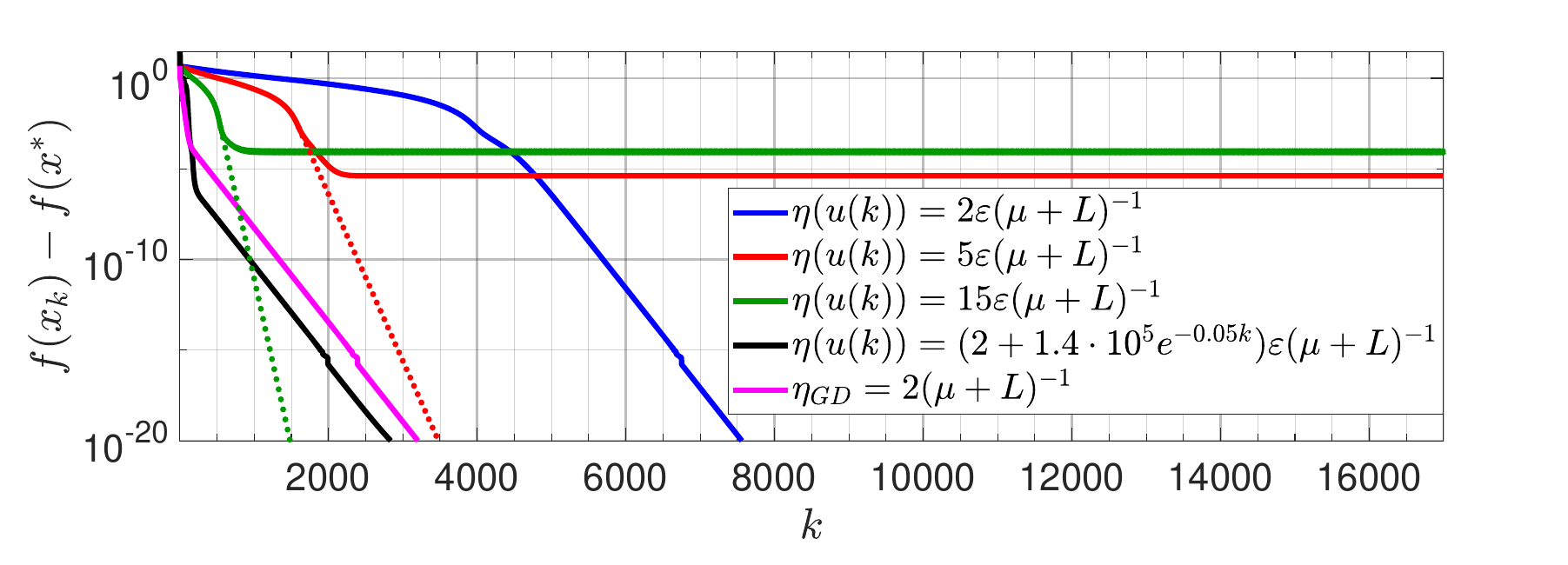}

    \vspace{-0.35cm}
    
    \caption{
    Simulation results for ~${f(x)\coloneqq \tfrac{1}{2}\tilde{x}^{\top}Q\tilde{x} +\sqrt{1 + \tilde{x}^{\top}Q\tilde{x}} - 1}$, where
    $x\in\R^{10}$,~${\tilde{x} \coloneqq x-x^{*}}$, $Q$ is a positive definite (random) matrix
    with minimal and maximal eigenvalue $0.003$ and $1$, and ~${\mu = 0.003}$ and ~${L = 2}$ \cite{schererebenbauer2025_tutorialIQC}.
    The algorithm parameters are ~${\beta = 0.1}$,~${\varepsilon = 10^{-3}}$ and $x_0,s_0$ are initialized by two times the all-ones and the zero vector, respectively. The graphs in blue, red, green, and black show RMSProp for different step size rules. The graph in magenta is the gradient descent algorithm (GD) with the optimal constant step size. 
    We observe that the magenta (GD) and the blue graph (RMSProp) have approximately the same convergence rate (in terms of the asymptotic slopes of the graphs). This is expected since the step size rule~${\eta(u(k))=2\varepsilon(\mu+L)^{-1}}$
    in \eqref{sys:x} for~${s_k = 0}$
    corresponds to the GD algorithm.
    We also observe that RMSProp suffers from a much slower transient phase since  $s_k>0$ causes initially smaller step sizes.
    However, one benefit of RMSProp is that we can increase the step size beyond the stability margin and obtain
    bounded iterates and higher convergence rates to a neighborhood of $x^*$ (shown by the red and green graphs).
    We see that a larger step size increases the error floor in the sense that $f(x_k)-f(x^*)$ becomes larger for
    large~$k$ $(k \geq 2500)$,
    but also the convergence rate (indicated by the dotted green and red slopes) towards the neighborhood defined by the error floor increases. Finally, \Cref{ISS-theorem} allows to design arbitrary time-varying step sizes and guarantees the boundedness of the iterates. For example, a time-varying step size (graph in black) provides comparatively better results in the transient phase and the same asymptotic convergence rate as for the blue graph. The corresponding step size was designed empirically, by choosing a comparatively large initial value that decays sufficiently fast over the iterations. Preliminary investigations suggest that similar step size rules can be designed for other objective functions within the same class, with the condition number of the objective function appearing to play an important role in the tuning process.}
    \label{fig:1}
\end{figure}

\begin{rmk}
The main result of this paper establishes global asymptotic stability and ISS for any~${\dst(\cdot) \in \Uu}$ of algorithm \eqref{sys:rmsprop}.
In the following we loosely discuss local exponential stability for quadratic objective functions based on the ISS-Lyapunov function. 

For the sake of simplicity, let~${x^{*} = 0}$. Further, let~${f(x)  \coloneqq x^{\top}Qx}$, with~${Q \succ 0}$. We consider the case of~${\dst(t) = 0}$ for all~${t \geq 0}$. We will now analyze the ISS-Lyapunov function \eqref{function:Lyapunov} and its difference by using \eqref{delta_V_ISS_ineq} under the assumption that~${\smax \leq \varepsilon}$ and~${\|x\|_{\infty} \leq \varepsilon}$. By using first order Taylor expansion of~${h}$, defined in \eqref{lyapunov:subfunctions}, around~${(x,s) = (0,0)}$, we get for \eqref{function:Lyapunov}
\begin{align*}
        \Vv(x,s) &= \gzero x^{\top}Qx+ \tfrac{2\gone}{3} \left(x^{\top}Qx\right)^{\frac{3}{2}} + \sum\limits_{i=1}^{d}\left(\tfrac{s_i}{\varepsilon} + \mathcal{O}\left(s_i^{\frac{3}{2}}\right)\right)\hspace{-0.07cm}.
    \end{align*}
In the following, we use inequalities resulting from the strong convexity of~$f$, as well as norm inequalities, to derive a lower bound of type~${\gmax \geq \hat{r}_1\|x\|_{\infty}}$, with~${\hat{r}_1 > 0}$. Recall~${\rho_1}$ and~${\rho_2}$ from Step 2.3 in the proof of \Cref{ISS-theorem}. Since~${\rho_1(\cdot,0) \geq \gzero}$ holds on~${[0,\infty)}$, and further, by using first order Taylor expansion of~${\rho_2}$ around~${(x,s) = (0,0)}$, it follows from~\eqref{delta_V_ISS_ineq} that
\begin{align*}
    \Delta \Vv & \leq -\hat{\tau}_1\smax -\hat{\tau}_2\|x\|_{\infty}^2 - \hat{\tau}_3\|x\|_{\infty}^3 + \mathcal{O}(\|x\|_{\infty}^4)
\end{align*}
for~${\|x\|_{\infty} \leq \varepsilon}$ and~${\|s\|_{\infty} \leq \varepsilon}$, and
where~${\hat{\tau}_1, \hat{\tau}_2}$ and~${\hat{\tau}_3}$ are some positive constants. 
By using further norm inequalities and suitable algebraic manipulations, we get
\begin{align*}
    \Delta \Vv &\leq -\hat{\tau}_4\Vv(x,s) + \mathcal{O}(\|x\|^4) + \mathcal{O}\left(\|s\|_1^{\frac{3}{2}}\right),
\end{align*}
where~${\hat{\tau}_4}$ is some suitable positive constant. Thus, it seems that one can establish local exponential stability with the established ISS Lyapunov function, corroborating the results in~\cite{dereich2025sharp_conv_rate_RMSProp}. However, a rigorous analysis is beyond the scope of this work.
\end{rmk}

\begin{rmk}
    In principle, there is a lot of flexibility to extend the proof to more general step size rules.
    For example, one could adapt algorithm \eqref{sys:rmsprop} to use different step size rules for individual components~${i \in \overline{1:d}}$, i.e.,~${\eta(\dst_i(\cdot))\coloneqq \etzero + \dst_i(\cdot)}$, where~${\dst_i(\cdot) \in \Uu}$. In that case, \Cref{ISS-theorem} would still hold and the proof would only require minor changes. The most noticeable change would be that~${\chi_{\Vv}}$ in~\eqref{DVV} would then depend on~${\|\dst\|_{\infty}}$.
    Furthermore, one could consider algorithms with a generalized step size ${\phi:\rrd \to (0, \infty)}$ of form
    \begin{align*}
            \snext_{i} \!=\! (1\!-\!\beta)s_{i} \!+\! \beta (\nabla_i f(x))^2, \hspace{0.15cm}
            \xnext_{i}\! =\! x_{i} \!-\! \tfrac{{\ett}}{\phi\left(\nabla_i f(x), s_i\right)}\nabla_i f(x)
    \end{align*}
    and study the stability of this class of algorithms under various conditions on~$\phi$, or incorporate state dependent step sizes rules for optimized transient phase behavior. 
\end{rmk}

\section{Discussion and Outlook}
In this paper we have established that the RMSProp algorithm is globally asymptotically stable for a suitable constant step size and input-to-state stable with respect to any time-varying bounded step size rule. This step size robustness property is beneficial for algorithm tuning since one can experiment with the step size
and balance convergence rate versus convergence accuracy.
Furthermore, the proposed Lyapunov function \eqref{function:Lyapunov} consists of two terms. 
The second term addresses the adaptive step size rule variable while
the first term is a $\gamma$-scaled version of a Lyapunov function for an algorithm without adaptive step size.
In our case, the first term is the $\gamma$-scaled objective function since RMSProp without adaptive step size corresponds to gradient descent.
We believe that this construction of a Lyapunov function can be applied to other algorithms with adaptive step size rules by 
replacing the first term by a Lyapunov function of the to-be-analyzed algorithm without adaptive step size. 
For future work, for example, one can utilize this construction to
analyze momentum-based optimization algorithms like Adam.

 \bibliographystyle{ieeetr}  
 \bibliography{sgd_ref.bib}

\end{document}